\documentclass[reqno]{amsart}
\usepackage{enumerate}
\usepackage[margin=1.2in]{geometry}
\usepackage{amssymb,url,color, booktabs,nccmath}
\usepackage{mathrsfs}
\usepackage{enumitem}
\usepackage{graphicx}
\usepackage{tikz}
\usepackage{amsthm}
\usepackage[normalem]{ulem}
\usepackage{color}
\usepackage[colorlinks=true]{hyperref}
\hypersetup{
    linkcolor=blue,          % color of internal links
    citecolor=red,        % color of links to bibliography
    filecolor=blue,      % color of file links
    urlcolor=cyan
}
\newif\ifshowrevisioncolors
\showrevisioncolorsfalse
\newif\ifshowbluerevisions
\showbluerevisionsfalse
\definecolor{darkergreen}{rgb}{0.0, 0.5, 0.0}
\newif\ifshowgreenrevisions
\showgreenrevisionstrue
\definecolor{newrevisionviolet}{rgb}{0.55,0.0,0.65}
\newif\ifshownewrevisions
\shownewrevisionsfalse

\ifshowrevisioncolors\else
  \showbluerevisionsfalse
  \showgreenrevisionsfalse
  \shownewrevisionsfalse
  \hypersetup{linkcolor=blue,citecolor=red,filecolor=black,urlcolor=black}
\fi

\makeatletter
\renewcommand{\tagform@}[1]{%
  \maketag@@@{\color{blue}(\ignorespaces#1\unskip\@@italiccorr)}}
\makeatother

\definecolor{mathcheckred}{rgb}{0.82,0.0,0.0}
\newcommand{\mathfix}[1]{#1}

\long\def\zhu#1{}
\long\def\zhao#1{}
\long\def\wang#1{}

\numberwithin{equation}{section}

\newcommand{\be}{\begin{eqnarray}}
\newcommand{\ee}{\end{eqnarray}}
\newcommand{\ce}{\begin{eqnarray*}}
\newcommand{\de}{\end{eqnarray*}}
\newtheorem{theorem}{Theorem}[section]
\newtheorem{lemma}[theorem]{Lemma}
\newtheorem{remark}[theorem]{Remark}
\newtheorem{definition}[theorem]{Definition}
\newtheorem{proposition}[theorem]{Proposition}
\newtheorem{Examples}[theorem]{Example}
\newtheorem{corollary}[theorem]{Corollary}
\newtheorem{assumption}{Assumption}[section]
\newtheorem*{theorem*}{Theorem}
\newtheorem*{lemma*}{Lemma}
\newtheorem*{remark*}{Remark}

\def\eps{\varepsilon}

\def\e{\mathrm{e}}

\def\p{\partial}

\def\[{{\Big[}}
\def\]{{\Big]}}
\def\<{{\langle}}
\def\>{{\rangle}}
\def\({{\Big(}}
\def\){{\Big)}}

\def\bx{{\mathbf{x}}}

\def\Law{{\mathord{{\rm Law}}}}
\def\dif{{\mathord{{\rm d}}}}

\def\no{\nonumber}
\def\={&\!\!=\!\!&}

\def\bB{{\mathbf B}}

\def\cB{{\mathcal B}}
\def\cC{{\mathcal C}}
\def\cD{{\mathcal D}}
\def\cE{{\mathcal E}}
\def\cF{{\mathcal F}}
\def\cG{{\mathcal G}}

\def\cL{{\mathcal L}}

\def\cR{{\mathcal R}}
\def\cS{{\mathcal S}}

\def\mN{{\mathbb N}}

\def\mR{{\mathbb R}}

\def\bB{{\mathbf B}}

\def\1{{\mathbf{1}}}

\def\sI{{\mathscr I}}

\def\geq{\geqslant}
\def\leq{\leqslant}

\def\le{\leqslant}

\def\div{\mathord{{\rm div}}}

\def\eps{\varepsilon}

\def\e{\mathrm{e}}

\def\p{\partial}

\def\[{{\Big[}}
\def\]{{\Big]}}
\def\<{{\langle}}
\def\>{{\rangle}}
\def\({{\Big(}}
\def\){{\Big)}}

\def\bx{{\mathbf{x}}}

\def\Law{{\mathord{{\rm Law}}}}
\def\dif{{\mathord{{\rm d}}}}

\def\no{\nonumber}
\def\={&\!\!=\!\!&}
\def\bt{\begin{theorem}}
\def\et{\end{theorem}}
\def\bl{\begin{lemma}}
\def\el{\end{lemma}}
\def\br{\begin{remark}}
\def\er{\end{remark}}
\def\bx{\begin{Examples}}
\def\ex{\end{Examples}}
\def\bd{\begin{definition}}
\def\ed{\end{definition}}
\def\bp{\begin{proposition}}
\def\ep{\end{proposition}}
\def\bc{\begin{corollary}}
\def\ec{\end{corollary}}

\def\geq{\geqslant}
\def\leq{\leqslant}

\def\le{\leqslant}

\def\div{\mathord{{\rm div}}}

 \def\R{\mathbb R}
 \def\R{\mathbb R}

\def\<{\langle} \def\>{\rangle}
\def\x{{\bf x}}

\allowdisplaybreaks

\begin{document}

\title[Nonlinear Kinetic Fokker-Planck Equations]
{Kinetic Fokker-Planck Equations with Nonlinear Diffusion} 

\author[Zimo Hao]{Zimo Hao}
\address[Z. Hao]{School of Mathematics and Statistics, Beijing Institute of Technology, Beijing 100081, China}
\email{zmhao@bit.edu.cn}

\author[Zhengyan Wu]{Zhengyan Wu}
\address[Z. Wu]{Department of Mathematics, Technische Universit\"at M\"unchen, Boltzmannstr. 3, 85748 Garching, Germany}
\email{wuzh@cit.tum.de} 

\author[Xicheng Zhang]{Xicheng Zhang}
\address[X. Zhang]{School of Mathematics and Statistics, Beijing Institute of Technology, Beijing 100081, China}
\email{XichengZhang@gmail.com}

\begin{abstract}
We consider the nonlinear kinetic Fokker--Planck equation on $\mR^{2d}$,
\begin{align*}
    \partial_t f=\Delta_v\Phi(f)-v\cdot\nabla_x f,
    \qquad f(0)=f_0,
\end{align*}
where $\Phi: (0,\infty)\to(0,\infty)$ is a continuous increasing function, and $0\leq f_0\in L^1\cap L^2$.
Under some assumptions on $\Phi$, in particular, including all
fast-diffusion powers $\Phi(\theta)=\theta^m$, $m\in(0,1)$, we construct
global nonnegative weak solutions, which admit a renormalized kinetic formulation and satisfy
$L^q$-dissipation estimates for $q\in[1,2]$ and quantitative
anisotropic Besov regularities, based on a kinetic-level-dependent estimate for
the Kolmogorov semigroup. Under the critical low-density growth
condition, the constructed solution preserves mass, and mass-preserving
renormalized kinetic solutions in
$L^\infty_{\mathrm{loc}}([0,\infty);L^1\cap L^2)$ is unqiue. Moreover, for unit-mass initial data, a strict
low-density growth condition yields on every finite time
interval a martingale solution to the associated density-dependent
kinetic stochastic differential equation with one-time marginal density
$f$.
\end{abstract}

\keywords{nonlinear Fokker--Planck equations; kinetic fast-diffusion equations; anisotropic Besov regularity; renormalized kinetic solutions; mass conservation}

\date{\today}

\maketitle

\setcounter{tocdepth}{1}

%\tableofcontents

\section{Introduction}

Nonlinear diffusion equations describe the evolution of densities whose
diffusivity depends on their local concentration. A classical model is
$$
  \partial_t u=\Delta\Phi(u),
$$
with the power nonlinearities $\Phi(u)=u^m$ giving the porous-medium equation
for $m>1$ and the fast-diffusion equation for $0<m<1$. These equations have
an extensive theory of existence, uniqueness, contraction, regularization,
and long-time behavior; see, for example, \cite{HP85,DK07,V07}.

For the linear heat equation, a simple degenerate extension is the kinetic
Kolmogorov equation
$$
  \partial_t f+v\cdot\nabla_x f=\Delta_v f,
  \qquad (x,v)\in\R^d\times\R^d.
$$
The diffusion is not uniformly elliptic in the full phase space, but the
operator is hypoelliptic: its interaction with free transport transfers
regularity from velocity to position \cite{Ko34,Bo02}. The natural
scaling is anisotropic,
$$
  (t,x,v)\longmapsto(r^2t,r^3x,rv),\qquad r>0,
$$
so regularity in position and velocity is measured at different orders.
For more general linear kinetic equations, many counterparts of classical
parabolic regularity estimates and well-posedness results remain available
under suitable assumptions, in function spaces adapted to the kinetic
structure; see \cite{HWZ20,CHM21,IS21,HZZZ24,HRZ25} and the references
therein.

This linear theory leads naturally to the question of whether kinetic
counterparts of nonlinear diffusion equations admit a comparable theory
of well-posedness and regularity. In this paper, we investigate the Cauchy
problem
\begin{equation}\label{intro:eq:main}
  \partial_t f=\Delta_v\Phi(f)-v\cdot\nabla_x f,
  \qquad f(0)=f_0,
\end{equation}
on $\R^{2d}$, with nonnegative initial data
$$
  f_0\in L^1(\R^{2d})\cap L^2(\R^{2d}).
$$
Our principal examples are $\Phi(r)=r^m$ with $0<m\leq1$.
Here the diffusion strength depends on the unknown density and may be
singular near the vacuum. The resulting interaction between nonlinear
diffusion and kinetic transport requires estimates that remain effective
at this low level of regularity.

The corresponding nonlinear kinetic theory is much less developed.
Recent contributions include the work of Brigati, Carlier, and Dolbeault
\cite{BCD26} on self-similar fundamental solutions and intermediate
asymptotics for the frictionless power-law equation. For the equation
with linear friction, Brigati, Carlier, Dolbeault, and Quattrocchi
\cite{BCDQ26} develop a variational construction, while Bouin, Dolbeault,
and Mellet \cite{BDM26} establish existence and uniqueness and study
diffusion limits. The present work was developed independently of and
concurrently with \cite{BDM26}. We give a more detailed comparison in
Section~\ref{intro:sec:related}.

We first state our assumptions and main results. They provide global
existence and quantitative anisotropic regularity for a structural class
of nonlinearities and unweighted $L^1\cap L^2$ initial data, without
boundedness or positive moment assumptions. Under a further growth
condition, we also obtain conservation of mass and $L^1$ contraction for
mass-preserving renormalized kinetic solutions, including the critical
power $m=1-1/d$ when $d\geq2$.

\subsection{Assumptions and main results}

We write the equation in the divergence form
\begin{align}\label{Aq1}
    \partial_t f
    =\operatorname{div}_v\bigl(a(f)\nabla_v f\bigr)
    -v\cdot\nabla_x f,
    \quad f(0)=f_0,
\end{align}
where
\begin{align*}
    a\in C^1((0,\infty);(0,\infty))
    \cap L^1_{\mathrm{loc}}([0,\infty)).
\end{align*}
Throughout the paper, we use
\begin{align}
    \Phi_a(\theta):=\int_0^\theta a(s)\dif s,
    \quad
    H_a(\theta):=\int_0^\theta\sqrt{a(s)}\dif s.
\end{align}
Thus, identifying the function $\Phi$ in \eqref{intro:eq:main} with $\Phi_a$,
equation \eqref{Aq1} is equivalently written as
$$
\partial_t f=\Delta_v\Phi_a(f)-v\cdot\nabla_xf.
$$

\begin{assumption}\label{A3}
Assume that there exist
$q_0,q_1\in(1,2]$, $\ell\in(0,1)$ and $\kappa_0>0$ such that
\begin{align}\label{As15}
    |(\ln a(\theta))'|
    \leq
    \kappa_0\bigl(\theta^{q_0-2}+\theta^{q_1-2}\bigr)
    \bigl(a(\theta)^\ell\wedge a(\theta)^{\ell/3}\bigr),
    \qquad \theta>0.
\end{align}
Assume in addition that, for some $\kappa_1>0$,
\begin{align}\label{As16}
    \int_0^\theta
    \bigl(a(s)^{-\ell}\vee a(s)^{-\ell/3}\bigr)\dif s\leq \kappa_1(\theta+\theta^2),
    \qquad \theta>0,
\end{align}
and that, for some $\kappa_2>0$, $\sigma_0\in[0,1]$, and
$\sigma_1\in[1,2)$,
\begin{align}\label{As17}
    0\leq \Phi_a(\theta)
    \leq \kappa_2(\theta^{\sigma_0}+\theta^{\sigma_1}),
    \qquad \theta\geq0.
\end{align}
\end{assumption}

\begin{Examples}\label{ex:power-coefficient}
Let $m\in(0,1)$ and
\begin{align*}
    a(\theta)=m\theta^{m-1},
    \qquad \theta>0.
\end{align*}
Then $\Phi_a(\theta)=\theta^m$. Assumption~\ref{A3} holds with
$\sigma_0=m$, $\sigma_1=1$, and, for example,
\begin{align*}
    q_0=1+(1-m)\ell,
    \qquad
    q_1=1+\tfrac{(1-m)\ell}{3},
\end{align*}
with constants depending only on $m$ and $\ell$.
The linear coefficient $a\equiv1$ also satisfies
Assumption~\ref{A3}.
\end{Examples}

We use the following two notions of solution and set
$\mR_+:=[0,\infty)$.

\begin{definition}[Weak solution]\label{def-weaksolution}
Let $0\leq f_0\in L^2$. A nonnegative
function $f\in C_w(\mR_+;L^2)$ is called a weak solution to \eqref{Aq1} if
$\Phi_a(f)\in L^1_{\mathrm{loc}}(\mR_+\times\mR^{2d})$ and, for every
$\varphi\in C_c^\infty(\mR^{2d})$ and every $t\geq0$,
\begin{align}\label{Weak1}
    \int f(t)\varphi
    &=\int f_0\varphi
      +\int_0^t\!\int \Phi_a(f)\Delta_v\varphi
      +\int_0^t\!\int f\,v\cdot\nabla_x\varphi.
\end{align}
\end{definition}

\br
Let $f_0\in L^1$, suppose that $a$ is bounded, and assume that
$\nabla_v f\in L^2_{\mathrm{loc}}(\mR_+\times\mR^{2d})$. Then
Proposition~\ref{prop:regular-L1-Lq} below shows that every weak solution
satisfies
\begin{align*}
    \|f(t)\|_{L^1}\leq\|f_0\|_{L^1}.
\end{align*}
\er

\begin{definition}[Renormalized kinetic solution]
\label{def-renormalized-kinetic}
Let $0\leq f_0\in L^2$. A nonnegative function $f$ satisfying
\begin{align*}
f\in C_w(\mR_+;L^2),\quad    \nabla_vH_a(f)\in L^2_{\mathrm{loc}}(\mR_+;L^2),
\end{align*}
is called a renormalized kinetic solution if there exists a nonnegative Radon
measure $\pi$ on $\mR_+\times\mR^{2d}\times\mR_+$ such that for every $T>0$,
$$
\pi([0,T]\times\mR^{2d}\times\mR_+)<\infty,
$$
and for every nonnegative
$h\in C_c(\mR_+\times\mR^{2d}\times\mR_+)$,
\begin{align}\label{controlKM}
    \int_{\mR_+\times\mR^{2d}}
    h(t,z,f)|\nabla_vH_a(f)|^2\dif z\dif t
    \leq
\int_{\mR_+\times\mR^{2d}\times\mR_+} h(t,z,\theta)\,\pi(\dif t,\dif z,\dif\theta),
\end{align}
and the kinetic functions
\begin{align*}
    \chi(t,z,\theta):=\1_{\{0<\theta<f(t,z)\}},\quad \chi_0(z,\theta):=\1_{\{0<\theta<f_0(z)\}}
\end{align*}
satisfy, for every
$\varphi\in C_c^\infty(\mR_+\times\mR^{2d}\times \mR_+)$,
\begin{align}\label{kinetic-formulation-distributional}
\begin{split}
    &\int_{\mR_+\times\mR^{2d}\times \mR_+}
    \chi\bigl(\partial_t\varphi+a(\theta)\Delta_v\varphi
    +v\cdot\nabla_x\varphi\bigr)\dif\theta\dif z\dif t\\
    &\quad-\int_{\mR_+\times\mR^{2d}\times\mR_+}
    \partial_\theta\varphi\,\dif\pi
    +\int_{\mR^{2d}}\!\int_{\mR_+}
    \chi_0(z,\theta)\varphi(0,z,\theta)\dif\theta\dif z=0.
\end{split}
\end{align}
\end{definition}
A renormalized kinetic solution is called mass preserving if
\begin{align*}
    \int_{\mR^{2d}}f(t,z)\dif z
    =
    \int_{\mR^{2d}}f_0(z)\dif z<\infty,
    \qquad t\geq 0.
\end{align*}

\begin{theorem}[Existence and regularity]\label{thm:main}
Suppose Assumption~\ref{A3} holds. For every nonnegative
$f_0\in L^1\cap L^2$, there exists a global weak solution
$f$ to \eqref{Aq1}, which is also a renormalized kinetic solution
and satisfies, for every $q\in[1,2]$ and every $t\geq0$,
\begin{align}\label{main-Lq}
\|f(t)\|_{L^q}^q
    +q(q-1)\int_0^t\!\int|\nabla_vH_{a,q}(f)|^2
    \leq\|f_0\|_{L^q}^q,
\end{align}
where the second term is understood as zero for $q=1$, and for $q\in(1,2]$,
\begin{align*}
    H_{a,q}(\theta):=\int_0^\theta s^{(q-2)/2}\sqrt{a(s)}\dif s.
\end{align*}
Moreover, let
$p\in [1,\frac{2d}{2d-\ell})$, where $\ell$ is the exponent in
Assumption~\ref{A3}. Then, for every $T>0$,
\begin{align}\label{main-besov}
\|f\|_{L^1(0,T;\bB_{p;\mathbf a}^\alpha)}
    \leq
    C_T\bigl(\|f_0\|_{L^1}+\|f_0\|_{L^2}^2\bigr),
\end{align}
where $\alpha:=2\ell-4d(1-\frac1p)$, $\mathbf a=(3,1)$, $\bB_{p;\mathbf a}^\alpha$ is the anisotropic Besov space defined in 
Definition \ref{Def1} below, and $C_T$ depends only on the  constants in Assumption~\ref{A3}
and on $d,p,T$.
\end{theorem}

The anisotropy $\mathbf a=(3,1)$ reflects the Kolmogorov scaling: one
velocity derivative has the same kinetic order as one third of a position
derivative. Thus \eqref{main-besov} quantifies the spatial regularity
transferred from velocity diffusion by free transport.

\begin{theorem}[Mass conservation and uniqueness]\label{thm:unique}
In addition to Assumption~\ref{A3}, suppose that
\begin{align}\label{critical-growth-condition}
    \sigma_0\geq1-\tfrac1d\quad\text{if }d\geq2,
    \qquad
    \sigma_0>0\quad\text{if }d=1.
\end{align}
Then the solution constructed in Theorem~\ref{thm:main} satisfies
\begin{align}
    \|f(t)\|_{L^1}=\|f_0\|_{L^1},
    \qquad t\geq0.
\end{align}
Moreover, it is the unique
mass-preserving renormalized kinetic solution in
$L^\infty_{\mathrm{loc}}([0,\infty);L^1\cap L^2)$ with initial datum
$f_0$. More generally, any two such solutions satisfy, for every
$t\geq0$,
\begin{align}\label{Contra}
    \|f_1(t)-f_2(t)\|_{L^1}
    \leq
    \|f_{1,0}-f_{2,0}\|_{L^1}.
\end{align}
\end{theorem}

\br
Estimate~\eqref{Contra} holds at every time without assuming
strong $L^1$ continuity a priori: the almost-everywhere-in-time
estimate \eqref{Contra} is extended to every $t\geq0$ by weak $L^2$ continuity and weak
lower semicontinuity of the local $L^1$ norm.
For the power coefficient in Example~\ref{ex:power-coefficient}, the
parameter $\ell$ in Theorem~\ref{thm:main} may be
chosen arbitrarily in $(0,1)$. Thus, the unique mass-preserving solution in Theorem~\ref{thm:unique} satisfies
$
    f\in\bigcap_{0<\alpha<2}
    L^1(0,T;\bB_{1;\mathbf a}^\alpha)
$
for any $T>0$.
\er
\br
Without essential difficulty, we can consider the following PDE:
$$
\partial_t f
    =\operatorname{div}_v\bigl(a(f)\nabla_v f\bigr)
    -v\cdot\nabla_x f-\div_v(b f),
    \qquad f(0)=f_0,
$$
where $b:\mathbb R_+\times\mathbb R^{2d}\to \mathbb R^d$ is Borel measurable and $b=b_0+b_1$ with $b_0\in L^\infty(\mathbb R_+\times\mathbb R^{2d})$ and 
$$
|b_1(t,z)-b_1(t,z')|\leq C|z-z'|.
$$
\er

\subsection{Probabilistic interpretation}
From a probabilistic perspective, 
when $f_0$ has unit mass and the solution $f$ preserves mass,
equation \eqref{intro:eq:main} also has a natural probabilistic interpretation.
We define
\begin{align}\label{density-diffusion-coefficient-intro}
    \Sigma(r):=
    \begin{cases}
        \Phi(r)/r,&r>0,\\
        0,&r=0.
    \end{cases}
\end{align}
Formally, \eqref{intro:eq:main} is the forward Kolmogorov equation of the following
distribution-dependent kinetic stochastic differential system
\begin{align}\label{formal-density-dependent-sde}
    \dif X_t&=V_t\dif t,\quad\
    \dif V_t=
    \sqrt{2\Sigma(f(t,X_t,V_t))}\,\dif B_t,
\end{align}
where $B$ is a standard $d$-dimensional Brownian motion and
$\Law(X_0,V_0)$ has density $f_0$. Indeed, if
\begin{align*}
    \Law(X_t,V_t)=f(t,z)\dif z,
\end{align*}
then by It\^o's formula formally, we have for
$\varphi\in C_c^\infty(\mR^{2d})$,
\begin{align*}
    \frac{\dif}{\dif t}
    \int_{\mR^{2d}}\varphi(z)f(t,z)\dif z
    &=
    \int_{\mR^{2d}}
    \left(
        v\cdot\nabla_x\varphi+
        \Sigma(f)\Delta_v\varphi
    \right)f\dif z
    =
    \int_{\mR^{2d}}
    \left(
        v\cdot\nabla_x\varphi\,f+
        \Phi(f)\Delta_v\varphi
    \right)\dif z,
\end{align*}
which is precisely the weak formulation of \eqref{intro:eq:main}.

\begin{definition}[Martingale solution]\label{def-martingale-solution}
Let $f:[0,T]\times\mR^{2d}\to[0,\infty)$ be Borel measurable and
suppose that
\begin{align*}
    \mu_t(\dif z):=f(t,z)\dif z,
    \qquad t\in[0,T],
\end{align*}
is a narrowly continuous family of probability measures. 
Assume in addition that
\begin{align*}
    \Phi_a(f)\in L^1(0,T;L^1_{\mathrm{loc}}(\mR^{2d})).
\end{align*}
Set
\begin{align}\label{density-diffusion-coefficient}
    \Sigma_f(t,z):=
    \begin{cases}
        \Phi_a(f(t,z))/f(t,z),&f(t,z)>0,\\
        0,&f(t,z)=0.
    \end{cases}
\end{align}
Let $\mathcal W_T:=C([0,T];\mR^{2d})$, let
$w_t=(w_t^X,w_t^V)$ be the canonical process, and let
$(\mathcal F_t)_{t\in[0,T]}$ be its canonical filtration.  A probability
measure $\mathbb P$ on $\mathcal W_T$ is a martingale solution associated
with $f$ if
\begin{align*}
    \mathbb P\circ w_t^{-1}(\dif z)=f(t,z)\dif z,
    \qquad t\in[0,T],
\end{align*}
and, for every $\varphi\in C_c^\infty(\mR^{2d})$,
\begin{align}\label{martingale-problem}
    M_t^\varphi
    &:=
    \varphi(w_t)-\varphi(w_0)
    -\int_0^t
    \left[
        w_s^V\cdot\nabla_x\varphi(w_s)
        +\Sigma_f(s,w_s)\Delta_v\varphi(w_s)
    \right]\dif s
\end{align}
is a continuous $(\mathcal F_t,\mathbb P)$-martingale.
\end{definition}

The martingale problem is well defined. Indeed, the transport term is
bounded on the support of $\varphi$, and
\begin{align*}
    \mathbb E_{\mathbb P}\int_0^T
    |\Sigma_f(s,w_s)\Delta_v\varphi(w_s)|\dif s
    =
    \int_0^T\!\int_{\{f>0\}}
    \Phi_a(f)|\Delta_v\varphi|\dif z\dif s<\infty.
\end{align*}
This identity also shows that the definition is independent of the chosen
Borel representative up to the relevant
$f(t,z)\dif t\dif z$-null sets.

\begin{corollary}[Probabilistic representation]\label{cor-probabilistic-representation}
Suppose Assumption~\ref{A3} holds and
\begin{align}\label{strict-probabilistic-growth}
    \sigma_0>1-\tfrac1d.
\end{align}
Let $f_0\in L^1\cap L^2$ be nonnegative with
$\int_{\mR^{2d}}f_0(z)\dif z=1$, and let $f$ be the solution constructed
in Theorem~\ref{thm:main}.  Then, for every $T>0$, there exists a
martingale solution $\mathbb P$ on $\mathcal W_T$ associated with $f$.
Equivalently, on a possibly enlarged filtered probability space, there
exist processes $X,V$ and a Brownian motion $B$ that solve
\eqref{formal-density-dependent-sde} weakly and satisfy
\begin{align*}
    \Law(X_t,V_t)=f(t,z)\dif z,
    \qquad t\in[0,T].
\end{align*}
\end{corollary}

The strict inequality in \eqref{strict-probabilistic-growth} is
used only to verify the weighted integrability required by the
superposition principle. The PDE uniqueness result itself includes the
critical endpoint when $d\geq 2$.

\subsection{Proof strategy}\label{intro:sec:strategy}
For existence, the main obstacle is that the basic $L^2$ energy estimate controls
$$
  \int_0^T\int_{\R^{2d}}|\nabla_v H_a(f)|^2\,\dif z\,\dif t,
$$
but yields no direct derivative in $x$. Regularity in position must therefore be
recovered through the interaction of velocity diffusion with free transport. For the linear kinetic equation with elliptic diffusion in velocity, this is handled by introducing an anisotropic Besov space and establishing the regularity estimate in this anisotropic direction, which automatically transfers regularity from velocity to position (see \cite{HWZ20,HZZZ24,HRZ25}).
The issue in the present paper is to quantify this transfer when the diffusion strength depends
on the density and is not bounded above and below uniformly.

We use a kinetic formulation, following the approach developed for
conservation laws and degenerate parabolic equations
\cite{LPT94,P02}. Introducing the kinetic level $\theta$ leads to the
family of linear operators
$$
  a(\theta)\Delta_v-v\cdot\nabla_x.
$$
The nonlinear dissipation is encoded by a nonnegative kinetic measure,
whose derivative in $\theta$ enters the equation (see Remark 3.2 below). Consequently, the
Duhamel representation involves differentiation of the semigroup kernel
with respect to $\theta$ (see Lemma 4.2 below). Our main analytical estimate controls both the
kernel and its first kinetic-level derivative, with explicit dependence
on $a(\theta)$ and $(\log a)'(\theta)$.

To construct solutions, we approximate $a$ by uniformly elliptic
coefficients for which the structural bounds are uniform. Applying the
semigroup estimate to the kinetic Duhamel formula gives the Besov bound
\eqref{main-besov}. Combined with time compactness, this yields strong
local convergence and allows us to identify the nonlinear diffusion term.
Lower semicontinuity preserves the domination property
\eqref{controlKM} of the limiting kinetic measure.

For the uniqueness and stability argument, we must control the behavior at infinity and near
the zero kinetic level. Under \eqref{critical-growth-condition}, a phase-space
cutoff estimate yields conservation of mass, including the critical
endpoint. Mass preservation in turn gives the small-level control of the
kinetic measure needed to remove the cutoff near $\theta=0$. In particular, we establish a
localized Kato's type inequality (c.f. \cite{K72}), which leads to global $L^1$ contraction after removal
of the phase-space cutoff. This argument uses the unweighted $L^1\cap L^2$
bounds and does not require a positive moment of the solution.

Finally, for probability-density initial data,
\eqref{strict-probabilistic-growth} gives the weighted integrability needed for
the superposition principle. This provides the martingale representation
in Corollary~\ref{cor-probabilistic-representation}.

\subsection{Related results}\label{intro:sec:related}

\begingroup
\emergencystretch=1em
For the frictionless power-law equation, Brigati, Carlier, and
Dolbeault \cite{BCD26} construct explicit self-similar fundamental solutions
for $1-1/d<m<1+1/d$, $m\neq1$.
Later, a variational construction for the equation with friction is given by Brigati, Carlier,
Dolbeault, and Quattrocchi \cite{BCDQ26}, who prove convergence of a
minimizing-movement scheme for $m\geq1$ and a maximal-dissipation
characterization for $1\leq m\leq3/2$.

% The red group marks the revisions to the comparison with BDM26.
% Remove \color{red} and the local \hypersetup to obtain the clean version.

In independent work developed concurrently with the present paper,
Bouin, Dolbeault, and Mellet \cite{BDM26} study nonlinear kinetic
equations with linear friction on both torus and $\mathbb R^{2d}$. Their whole-space fast-diffusion result
\cite[Theorem~1.3]{BDM26} gives global existence and uniqueness for
$$
  \partial_t f+v\cdot\nabla_x f
    =\Delta_v f^m+\div_v(vf),
  \qquad (x,v)\in\R^{2d},
$$
when $1-1/d<m<1$, for nonnegative $L^1\cap L^\infty$ initial data
with suitable positive moments in both position and velocity.

For the frictionless equation \eqref{intro:eq:main}, we have global
existence for arbitrary nonnegative $L^1\cap L^2$ initial data under
Assumption~\ref{A3}, without boundedness or moment assumptions
(Theorem~\ref{thm:main}). This includes every power
$\Phi(r)=r^m$ with $0<m\leq1$. We also have quantitative anisotropic
Besov estimates; for these powers, the solution belongs to
$L^1(0,T;\mathbf B^{\alpha}_{1;{\bf a}})$ for every $0<\alpha<2$.
Under the additional growth condition \eqref{critical-growth-condition},
we prove mass conservation and $L^1$ contraction, including the critical
endpoint $m=1-1/d$ when $d\geq2$
(Theorem~\ref{thm:unique}).
The contraction result holds in the class of mass-preserving
renormalized kinetic solutions whose kinetic measure satisfies the
energy-domination property \eqref{controlKM}.

\subsection{Contributions}\label{intro:sec:contributions}

Our main contributions can be summarized as follows.
\begin{enumerate}
\renewcommand{\labelenumi}{\arabic{enumi}.}
\item \textbf{Global existence and quantitative anisotropic regularity.}
Under Assumption~\ref{A3}, we construct global nonnegative
weak solutions for arbitrary nonnegative $L^1\cap L^2$ initial data,
without boundedness or positive moment assumptions. The solutions admit
a renormalized kinetic formulation and satisfy the $L^q$ dissipation
bounds \eqref{main-Lq} and the anisotropic Besov estimate
\eqref{main-besov}. The coefficient class includes all powers
$\Phi(r)=r^m$ with $0<m\leq1$. A key analytical ingredient is a
parameter-dependent dyadic estimate for the kinetic semigroup and its
first derivative in the kinetic level, which provides the regularity
and compactness needed for the construction.

\item \textbf{Mass conservation and $L^1$ stability at the critical threshold.}
Under the additional growth condition \eqref{critical-growth-condition}, we
prove mass conservation for the constructed solutions and $L^1$
contraction, hence uniqueness, in the class of mass-preserving
renormalized kinetic solutions. For power nonlinearities, this covers
$1-1/d\leq m\leq1$ when $d\geq2$, including the critical endpoint,
and every $0<m\leq1$ when $d=1$. 
The argument uses the $L^1$ and $L^2$ estimates together with the domination property of the kinetic measure, without additional moment assumptions.

\item \textbf{Probabilistic representation.}
For initial data that are probability densities and under the strict
growth condition \eqref{strict-probabilistic-growth}, we represent the PDE
solution as the family of one-time marginals of a martingale solution
to the density-dependent stochastic system
\eqref{formal-density-dependent-sde}. This is an existence
result for the stochastic system; uniqueness is not asserted.
\end{enumerate}

\subsection*{Structure}
The paper is organized as follows. Section~\ref{sec:semigroup} proves the
parameter-dependent kinetic semigroup estimate. Section~\ref{sec:kinetic}
develops the relations between weak and renormalized kinetic
solutions, including mass estimates, kinetic-level cutoff properties, and
recovery of the weak equation. Section~\ref{sec:elliptic} treats uniformly
elliptic coefficients and establishes the uniform Besov estimate.
Section~\ref{sec:existence} constructs solutions for general coefficients
satisfying Assumption~\ref{A3} and proves the probabilistic representation.
Section~\ref{sec:uniqueness} proves the localized Kato inequality and the
$L^1$-contraction principle. The auxiliary coefficient-approximation and
Duhamel arguments are collected in the appendices.

\section{A parameter-dependent kinetic semigroup}\label{sec:semigroup}

We recall the anisotropic Littlewood--Paley decomposition adapted to
kinetic scaling. Throughout the paper, we set
\begin{align*}
    \mathbf a=(3,1),
    \qquad
    |z|_{\mathbf a}:=|x|^{1/3}+|v|,\ \ z=(x,v).
\end{align*}
Let $\psi\in C_c^\infty(\mR^{2d})$ be a nonnegative even function
with $\psi=1$ on
$\{|z|_{\mathbf a}\leq1\}$ and $\psi=0$ on
$\{|z|_{\mathbf a}\geq4/3\}$. Set
\begin{align*}
    \phi_{-1}^{\mathbf a}(\xi)&:=\psi(\xi),\ \
    \phi_j^{\mathbf a}(\xi)
    :=\psi(2^{-\mathbf a(j+1)}\xi)-\psi(2^{-\mathbf a j}\xi),
    \qquad j\geq0,
\end{align*}
where, for $\xi=(\xi_x,\xi_v)$,
\begin{align}
    2^{-\mathbf a j}\xi:=(2^{-3j}\xi_x,2^{-j}\xi_v).
\end{align}
Let $\cS$ be the Schwartz space on $\mR^{2d}$, and let $\cS'$ be
its topological dual, the space of tempered distributions. We denote the
Fourier transform and its inverse on $\cS'$ by $\mathcal F$ and
$\mathcal F^{-1}$, respectively. For $f\in\cS'$, define the block
operators by
\begin{align}\label{RR1}
\mathcal{R}_j^{\mathbf a} f:=\mathcal{F}^{-1}(\phi_j^{\mathbf a}\mathcal{F}f)=(\cF^{-1}\phi_j^{\mathbf a})\ast f,\ \ j\geq -1.
\end{align}

\begin{definition}\label{Def1}
For $\alpha\in\mR$ and $p\in[1,\infty]$, the anisotropic Besov space
$\bB_{p;\mathbf a}^{\alpha}$ is defined by 
\begin{align*}
\bB_{p;\mathbf a}^{\alpha}:=\Big\{f\in\mathcal{S}':\ \|f\|_{\bB_{p;\mathbf a}^{\alpha}}:=\sup_{j\geq -1}\left
(2^{\alpha j}\|\mathcal{R}_j^{\mathbf a} f\|_{L^p}\right)<\infty\Big\}.
\end{align*}
\end{definition}
We use the standard anisotropic Bernstein inequalities: for
$m,n\in\mN_0$, $1\leq p\leq q\leq\infty$, and $j\geq-1$,
\begin{align}\label{eq:clt-lp-gradient}
    \|\nabla_x^m\nabla_v^n\cR_j^{\mathbf a}f\|_{L^q}
    \leq
    C2^{[3m+n+4d(1/p-1/q)]j}
    \|\cR_j^{\mathbf a}f\|_{L^p}.
\end{align}
We also use without further comment that multiplication by a fixed smooth
compactly supported function is bounded on every
$\bB_{p;\mathbf a}^{\alpha}$. These facts follow from the usual
Littlewood--Paley argument; see, for example, \cite[Lemma 2.11]{HZZZ24}.

\subsection{The kinetic Gaussian kernel}

Let $a_\theta:(0,\infty)\to(0,\infty)$ be a $C^1$ function. For fixed
$\theta>0$, consider
\begin{align*}
    X_t^\theta=-\sqrt{2a_\theta}\int_0^tB_s\dif s,
    \qquad
    V_t^\theta=\sqrt{2a_\theta}B_t,
\end{align*}
where $B$ is a standard $d$-dimensional Brownian motion. The corresponding semigroup is
\begin{align}\label{Semi1}
    P_t^\theta f(x,v)
    :=\mathbb E f(x-tv+X_t^\theta,v+V_t^\theta)
    =(\Gamma_t p_t^\theta)*(\Gamma_t f)(x,v),
\end{align}
where
\begin{align}
    \Gamma_t f(x,v):=f(x-tv,v)
\end{align}
and
\begin{align}
    p_t^\theta(x,v)
    =\left(\frac{4\pi^2a_\theta^2t^4}{3}\right)^{-d/2}
    \exp\left(-\frac{3|x|^2+|3x+2tv|^2}{4a_\theta t^3}\right).
\end{align}
Thus $P_t^\theta$ is generated by
$a_\theta\Delta_v-v\cdot\nabla_x$.

We need the following estimate for establishing the compactness in the proof of existence of weak solutions.

\begin{lemma}\label{lem-kineticsemigroup-es}
For $k\in\{0,1\}$, $\lambda\geq0$, and $p\in[1,\infty]$, there exists a constant
$C=C(d,\lambda,p,k)>0$ such that, for every $t,\theta>0$ and $j\geq0$,
\begin{align}\label{eq:kinetic-semigroup-block-estimate}
    \left\|
        \cR_j^{\mathbf a}\partial_\theta^k(\Gamma_t p_t^\theta)
    \right\|_{L^p}
    \leq
    C2^{4dj(1-1/p)}
    |(\ln a_\theta)'|^k
    \left(
        \frac{2^{-2j}t^{-1}}
        {a_\theta\wedge a_\theta^{1/3}}
    \right)^\lambda.
\end{align}
Here and below, when $k=0$, the factor $|(\ln a_\theta)'|^0$ is understood to be equal to one.
\end{lemma}

\begin{proof}
Let $\bar p_t$ denote the density corresponding to $a_\theta=1$. By
scaling, we have
$$
    \Gamma_t p_t^\theta(x,v)
    =t^{-2d}a_\theta^{-d}
    F_0\left(a_\theta^{-1/2}t^{-3/2}x,
             a_\theta^{-1/2}t^{-1/2}v\right),
$$
where $F_0:=\Gamma_1\bar p_1$ is a Schwartz function. Differentiating in
$\theta$ gives
\begin{align}\label{Fs1}
    \partial_\theta(\Gamma_t p_t^\theta)(x,v)
    =t^{-2d}a_\theta^{-d}(\ln a_\theta)'
    F_1\left(a_\theta^{-1/2}t^{-3/2}x,
             a_\theta^{-1/2}t^{-1/2}v\right),
\end{align}
where
\begin{align*}
    F_1=-dF_0-\tfrac12(x\cdot\nabla_x+v\cdot\nabla_v)F_0
\end{align*}
is again a Schwartz function. We treat both cases simultaneously and write $F_k$ for
$F_0$ or $F_1$.

Let $\check{\phi}_j^{\mathbf{a}}(z)=(\cF^{-1}\phi^{\mathbf a}_j)(z)$ and
set $h:=2^{-j}t^{-1/2}$. Since
$\check\phi_j^{\mathbf a}(z)
    =2^{4dj}\check\phi_0^{\mathbf a}(2^{\mathbf a j}z),$
by a change of variables, we have
\begin{align}\label{scaled-block-norm}
\begin{split}
    &\left\|\cR_j^{\mathbf a}
    \partial_\theta^k(\Gamma_t p_t^\theta)\right\|_{L^p}=
    t^{-2d}a_\theta^{-d}|(\ln a_\theta)'|^k2^{-4dj/p}
    \left\|\check\phi_0^{\mathbf a}*
    F_k(a_\theta^{-1/2}h^{\mathbf a}\cdot)\right\|_{L^p},
\end{split}
\end{align}
where $h^{\mathbf a}z=(h^3x,hv)$.
Fix $m\in\mN_0$. Since $\phi_0^{\mathbf a}$ is supported away from the
origin, we have
\begin{align*}
    K_m:=\cF^{-1}\left[
        (|\xi_x|^2+|\xi_v|^2)^{-m}\phi_0^{\mathbf a}
    \right]\in\cS(\mR^{2d})
\end{align*}
and
$\check\phi_0^{\mathbf a}=(-\Delta_x-\Delta_v)^mK_m$. Hence, by Young's
inequality and the chain rule,
\begin{align*}
    \left\|\check\phi_0^{\mathbf a}*
    F_k(a_\theta^{-1/2}h^{\mathbf a}\cdot)\right\|_{L^p}
    &\leq \|K_m\|_{L^p}
    \left\|(-\Delta_x-\Delta_v)^m
    F_k(a_\theta^{-1/2}h^{\mathbf a}\cdot)\right\|_{L^1}\\
    &\leq \|K_m\|_{L^p} a_\theta^{d-m}h^{-4d}
    \sum_{r=0}^m
        \binom{m}{r}
        h^{6r+2(m-r)}
        \left\|
            \Delta_x^r\Delta_v^{m-r}F_k
        \right\|_{L^1}\\
    &\leq
    C_{d,m,p,k}a_\theta^{d-m}h^{-4d}(h^2+h^6)^m.
\end{align*}
Substituting this into \eqref{scaled-block-norm} and using
$t^{-2d}h^{-4d}=2^{4dj}$, we obtain
\begin{align}\label{integer-block-bound}
    \left\|\cR_j^{\mathbf a}
    \partial_\theta^k(\Gamma_t p_t^\theta)\right\|_{L^p}
    \leq
    C_{d,m,p,k}2^{4dj(1-1/p)}|(\ln a_\theta)'|^k
    \left[a_\theta^{-1}(h^2+h^6)\right]^m.
\end{align}
Define
$
    R:=h^2/(a_\theta\wedge a_\theta^{1/3}).
$
It is easy to see that
$$
    a_\theta^{-1}(h^2+h^6)\leq R+R^3.
$$
When $R\geq1$, the case $m=0$ in \eqref{integer-block-bound} gives the
required estimate. For $R<1$, choose  
$m\geq\lambda$ so that
\begin{align*}
    [a_\theta^{-1}(h^2+h^6)]^m
    \leq 2^mR^m\leq2^mR^\lambda.
\end{align*}
Substituting this into \eqref{integer-block-bound},
we get \eqref{eq:kinetic-semigroup-block-estimate} for every $\lambda\geq0$.
\end{proof}

\subsection{Mollifiers and cutoffs}
We use the following mollifiers and cutoff functions.
\begin{itemize}
\item Let $\varrho\in C_c^\infty(\mR^d)$ be a nonnegative even probability
density supported in the unit ball and set
\begin{align}\label{Mo8}
    \varrho_\eps(y):=\eps^{-d}\varrho(y/\eps).
\end{align}
For phase-space regularization we use
\begin{align}\label{phase-mollifier}
\rho(x,v):=\varrho(x)\varrho(v),\qquad
    \rho_\eps(x,v):=\eps^{-3d}\rho(x/\eps,v/\eps^2)
    =\varrho_\eps(x)\varrho_{\eps^2}(v).
\end{align}
The relative scales in \eqref{phase-mollifier} ensure that the commutator
with $v\cdot\nabla_x$ is of order $\eps$.

\item Let $\gamma\in C_c^\infty(\mR^d;[0,1])$ satisfy $\gamma=1$ on $B_1$ and
$\gamma=0$ outside $B_2$, and set
\begin{align}\label{Ga1}
    \gamma_R(x):=\gamma(x/R), \ \ R>0.
\end{align}

\item For $\delta\in(0,1/4)$, fix a function
$\zeta_\delta\in C_c^\infty((0,\infty);[0,1])$ such that
\begin{align}\label{kinetic-cutoff}
\begin{split}
    &\zeta_\delta=1\quad\text{on }[2\delta,\delta^{-1}],
    \qquad
    \operatorname{supp}\zeta_\delta
    \subset[\delta/4,\delta^{-1}+1],\\
    &|\zeta_\delta'(\theta)|
    \leq 2\left(
       \delta^{-1}\1_{[\delta/4,2\delta]}(\theta)
       +\1_{[\delta^{-1},\delta^{-1}+1]}(\theta)
    \right),\\
    &\delta^{-1}\1_{[\delta/2,\delta]}(\theta)
    \leq \zeta_\delta'(\theta)
       +2\,\1_{[\delta^{-1},\delta^{-1}+1]}(\theta).
\end{split}
\end{align}
The last inequality is understood after choosing $\zeta_\delta$ increasing
on $[\delta/4,2\delta]$ and decreasing only on the upper transition
interval.  As $\delta\downarrow0$, the lower transition collapses to
$\theta=0$ and the upper transition moves to infinity; hence
$\zeta_\delta(\theta)\to1$ for every fixed $\theta>0$.  Its derivative is
supported only in these two transition regions, which are precisely the
error regions in the subsequent kinetic estimates.  The geometry of the
cutoff is summarized in Figure~\ref{fig:kinetic-cutoff}.

\begin{figure}[htbp]
\centering
\begin{tikzpicture}[x=0.95cm,y=1.05cm,>=stealth,font=\small]
    \fill[gray!12] (1.20,0) rectangle (4.30,1.70);
    \fill[gray!25] (2.10,0) rectangle (3.10,1.70);
    \fill[gray!12] (8.10,0) rectangle (10.20,1.70);

    \draw[->] (0.35,0)--(5.72,0);
    \draw[->] (6.18,0)--(11.35,0) node[right] {$\theta$};
    \draw[->] (0.35,0)--(0.35,2.35)
        node[above] {$\zeta_\delta(\theta)$};
    \draw[densely dashed,gray] (0.35,1.70)--(10.20,1.70);
    \node[left] at (0.35,1.70) {$1$};

    \draw[very thick]
        (0.35,0)--(1.20,0)
        .. controls (1.60,0) and (1.82,0.15) .. (2.10,0.45)
        .. controls (2.38,0.88) and (2.72,1.32) .. (3.10,1.52)
        .. controls (3.43,1.68) and (3.82,1.70) .. (4.30,1.70)
        --(5.72,1.70);
    \draw[very thick]
        (6.18,1.70)--(8.10,1.70)
        .. controls (8.65,1.70) and (9.65,0) .. (10.20,0)
        --(11.20,0);

    \draw (5.76,-0.10)--(5.91,0.10);
    \draw (5.99,-0.10)--(6.14,0.10);
    \draw (5.76,1.60)--(5.91,1.80);
    \draw (5.99,1.60)--(6.14,1.80);
    \node[below] at (5.95,-0.12) {\scriptsize not to scale};

    \foreach \x/\lab in {
        1.20/{\delta/4},
        2.10/{\delta/2},
        3.10/{\delta},
        4.30/{2\delta},
        8.10/{\delta^{-1}},
        10.20/{\delta^{-1}+1}}
    {
        \draw (\x,0.07)--(\x,-0.07);
        \node[below] at (\x,-0.07) {$\lab$};
    }

    \draw[line width=1.2pt] (2.10,-0.55)--(3.10,-0.55);
    \draw (2.10,-0.48)--(2.10,-0.62);
    \draw (3.10,-0.48)--(3.10,-0.62);
    \node[below] at (2.60,-0.58)
        {$\zeta_\delta'\geq\delta^{-1}$};

    \node[align=center] at (2.70,2.15)
        {lower transition\\[-1pt]
        $\lvert\zeta_\delta'\rvert\leq2\delta^{-1}$};
    \node at (6.90,2.02) {$\zeta_\delta=1$};
    \node[align=center] at (9.15,2.15)
        {upper transition\\[-1pt]
        $-2\leq\zeta_\delta'\leq0$};
    \node at (0.78,0.35) {$\zeta_\delta=0$};
    \node at (10.70,0.35) {$\zeta_\delta=0$};
\end{tikzpicture}
\caption{Schematic profile of the kinetic cutoff $\zeta_\delta$.}
\label{fig:kinetic-cutoff}
\end{figure}
\end{itemize}

\section{Weak solutions and renormalized kinetic solutions}\label{sec:kinetic}

We collect here the relations between the two solution concepts and the
properties of the kinetic measures needed in the uniqueness argument.

\subsection{Regular weak solutions and mass conservation}

\bp\label{prop:regular-renormalization}
Suppose that $a$ is bounded.
Let $f$ be a weak solution in the sense of Definition~\ref{def-weaksolution}
and satisfy
$
\nabla_v f\in
L^2_{\mathrm{loc}}([0,\infty)\times\mR^{2d}).
$
\begin{itemize}
\item[(i)] Let $\beta\in C^2(\mR)$ be nonnegative and
suppose that, for some $\kappa>0$,
\begin{equation}\label{Au1}
    |\beta(s)|\le \kappa(1+|s|^2),\quad |\beta'(s)|\le \kappa(1+|s|),    \quad    |\beta''(s)|\leq \kappa.
\end{equation}
Then, for every $\varphi\in C^\infty_c(\mR^{2d})$ and $t>0$,
\begin{align}\label{Au2}
\begin{split}
    \int \beta(f(t))\varphi
    &=
    \int\beta(f_0)\varphi
    -
    \int_0^t\!\!\!\int
    a(f)
    (\nabla_v \beta(f)\cdot\nabla_v\varphi)\\
   &- \int_0^t\!\!\!\int
    \beta''(f)a(f)
    |\nabla_v f|^2\varphi+
    \int_0^t\!\!\!\int\beta(f)
    (v\cdot\nabla_x\varphi).
\end{split}
\end{align}
\item[(ii)] For $\theta>0$, define
$\chi_\theta(t,z):=\1_{\{f(t,z)>\theta\}}$.
Then, for every $\varphi\in C_c^\infty(\mR^{2d}\times(0,\infty))$,
writing $\varphi_\theta(z):=\varphi(z,\theta)$, and every $t>0$,
\begin{align}\label{MC-kinetic-solution1}
\begin{split}
\int^\infty_0\!\!\!\int \chi_\theta(t)\varphi_\theta\dif\theta
&=\int^\infty_0\!\!\!\int \chi_\theta(0)\varphi_\theta\dif\theta
-\int^t_0\!\!\!\int |\nabla_v f|^2(a(\theta)\p_\theta\varphi_\theta)|_{\theta=f}\\
&\quad+\int^\infty_0\!\!\int^t_0\!\!\!\int \chi_\theta\Big[a(\theta)\Delta_v\varphi_\theta+v\cdot\nabla_x\varphi_\theta\Big]\dif\theta.
\end{split}
\end{align}		
\end{itemize}
\ep
\begin{proof}
(i) Let $\rho_\eps$ be the mollifiers defined in \eqref{phase-mollifier}. Define
$$
f_\eps:=f*\rho_\eps .
$$
Using $\rho_\eps(z-\cdot)$ as a test function in \eqref{Weak1} yields the mollified equation
\begin{equation*}
    \partial_t f_\eps
    =
    \div_v(\rho_\eps*\nabla_v\Phi_a(f))
    -
    v\cdot\nabla_x f_\eps
    +
    [v\cdot\nabla_x,\rho_\eps*]f ,
\end{equation*}
where the commutator is defined by
\begin{equation*}
    [v\cdot\nabla_x,\rho_\eps*]f(t,z)
    =
    \int_{\mathbb R^{2d}}
w\cdot\nabla_y\rho_\eps(\bar z)\,
    f(t,z-\bar z)\,\dif \bar z,\ \ \bar z=(y,w) .
\end{equation*}
Since $\beta''$ is bounded and
$\beta'$ has at most linear growth, by the Sobolev chain rule in time, we have
\begin{equation*}
    \partial_t\beta(f_\eps)
    =
    \beta'(f_\eps)
    \Big(
\div_v(\rho_\eps*\nabla_v\Phi_a(f))
        -
        v\cdot\nabla_x f_\eps
        +
        [v\cdot\nabla_x,\rho_\eps*]f
    \Big).
\end{equation*}
Let $\varphi\in C^\infty_c(\mR^{2d})$.
Multiplying the identity by $\varphi$, integrating over
$[0,t]\times\mR^{2d}$, and by the integration by parts in $v$ and $x$, we have
\begin{align*}
    \int  \beta(f_\eps(t))\varphi
    &=
    \int \beta(f_0*\rho_\eps)\varphi-
    \int_0^t\!\!\!\int
    \beta''(f_\eps)(\rho_\eps*\nabla_v\Phi_a(f))\cdot\nabla_v f_\eps\,\varphi\\
    &\quad-
    \int_0^t\!\!\!\int
    \beta'(f_\eps)(\rho_\eps*\nabla_v\Phi_a(f))\cdot\nabla_v\varphi\\
    &\quad+
    \int_0^t\!\!\!\int
    \beta(f_\eps)v\cdot\nabla_x\varphi +
    \int_0^t\!\!\!\int
    \beta'(f_\eps)[v\cdot\nabla_x,\rho_\eps*]f\,\varphi.
\end{align*}
We pass to the limit $\eps\to0$. 
Since $a$ is bounded, the map $\Phi_a$ is Lipschitz on
$[0,\infty)$. The Sobolev chain rule therefore gives
\begin{align*}
    \nabla_v\Phi_a(f)=a(f)\nabla_vf\in L^2_{\mathrm{loc}}.
\end{align*}
Note that on every compact cylinder $Q$,
\begin{align*}
    f_\eps\to f,\quad \nabla_vf_\eps\to\nabla_vf,\quad\
    \nabla_v(\Phi_a(f))_\eps\to\nabla_v\Phi_a(f) \mbox{ strongly in $L^2(Q)$.}
\end{align*}
 Since $\beta'$ is globally Lipschitz and has at
most linear growth, while $\beta''$ is bounded and continuous, these
convergences, H\"older's inequality, and dominated convergence give the
limits of all terms except the commutator term. If
$\operatorname{supp}\varphi\subset B_R$, then, for sufficiently small
$\eps$,
\begin{align*}
 \|[v\cdot\nabla_x,\rho_\eps*]f\|_{L^2((0,t)\times B_R)}
 \leq C\eps\|f\|_{L^2((0,t)\times B_{R+1})}.
\end{align*}
Moreover,
$\|\beta'(f_\eps)\varphi\|_{L^2((0,t)\times\mR^{2d})}$ is uniformly
bounded. Hence the commutator term converges to zero, and the preceding
identity gives \eqref{Au2} by taking limits $\eps\to0$.

(ii) 
Let $\varphi\in C^2_c(\mR^{2d}\times(0,\infty))$, and choose
$0<c<C<\infty$ such that its $\theta$-support is contained in $[c,C]$.
Let $\varrho\in C_c^\infty((-1,1))$ be a nonnegative even probability
density and set
$
    \varrho_\eps(r):=\eps^{-1}\varrho(r/\eps).
$
For each $\theta\geq0$, set
$$
\beta^\eps_\theta(s):=\int_{-\infty}^{s-\theta}\varrho_\eps(r)\dif r.
$$
By \eqref{Au2}, we have
\begin{align*}
\int \beta^\eps_\theta(f(t))\varphi_\theta&=\int \beta^\eps_\theta(f_0)\varphi_\theta
-\int^t_0\!\!\!\int(\beta^\eps_\theta)'(f)\, a(f)(\nabla_v f\cdot\nabla_v\varphi_\theta)\\
&-\int^t_0\!\!\!\int(\beta^\eps_\theta)''(f)\, a(f)|\nabla_v f|^2\varphi_\theta
+\int_0^t\!\!\!\int \beta^\eps_\theta(f)(v\cdot\nabla_x \varphi_\theta).
\end{align*}
Integrating both sides with respect to $\theta$ and letting
$\eps\to0$, we obtain
$$
\lim_{\eps\to 0}\int^\infty_0\!\!\!\int \beta^\eps_\theta(f(t))\varphi_\theta\dif\theta =\int^\infty_0\!\!\!\int \1_{\{f(t)>\theta\}}\varphi_\theta\dif\theta
=\int^\infty_0\!\!\!\int \chi_\theta(t)\varphi_\theta\dif\theta,
$$
and
$$
\lim_{\eps\to 0}\int^\infty_0\!\!\!\int^t_0\!\!\int \beta^\eps_\theta(f)(v\cdot\nabla_x \varphi_\theta)\dif\theta
=\int^\infty_0\!\!\!\int^t_0\!\!\int \chi_\theta\,(v\cdot\nabla_x \varphi_\theta)\dif\theta.
$$
On the other hand, by Fubini's theorem and integration by parts with respect
to $\theta$, we have
\begin{align*}
\int^\infty_0\!\!\!\int^t_0\!\!\int(\beta^\eps_\theta)''(f)\, a(f)|\nabla_v f|^2\varphi_\theta\dif\theta
&=\int^t_0\!\!\int\!\!\int^\infty_0\varrho_\eps(f-\theta)\, a(f)|\nabla_v f|^2\p_\theta\varphi_\theta\dif\theta.
\end{align*}
Since $\varrho_\eps(s)\to\delta_{\{0\}}(\dif s)$ as $\eps\to0$, we get
\begin{align*}
\lim_{\eps\to 0}\int^\infty_0\!\!\!\int^t_0\!\!\int(\beta^\eps_\theta)''(f)\, a(f)|\nabla_v f|^2\varphi_\theta\dif\theta
=\int^t_0\!\!\int a(f)|\nabla_v f|^2(\p_\theta\varphi_\theta)|_{\theta=f}.
\end{align*}
Similarly,
\begin{align*}
&\lim_{\eps\to 0}\int^\infty_0\!\!\!\int^t_0\!\!\!\int(\beta^\eps_\theta)'(f)\, a(f)(\nabla_v f\cdot\nabla_v\varphi_\theta)\dif\theta
=\int^t_0\!\!\!\int
\left.a(\theta)\nabla_vf\cdot\nabla_v\varphi_\theta\right|_{\theta=f}\\
&\quad=\lim_{\eps\to 0}\int^\infty_0\!\!\int^t_0\!\!\!\int (\beta^\eps_\theta)'(f)(\nabla_vf\cdot a(\theta)\nabla_v\varphi_\theta)\dif\theta
=\lim_{\eps\to 0}\int^\infty_0\!\!\int^t_0\!\!\!\int (\nabla_v\beta^\eps_\theta(f)\cdot a(\theta)\nabla_v\varphi_\theta)\dif\theta\\
&\quad=-\lim_{\eps\to 0}\int^\infty_0\!\!\int^t_0\!\!\!\int a(\theta)\beta^\eps_\theta(f)\Delta_v\varphi_\theta\dif\theta
=-\int^\infty_0\!\!\int^t_0\!\!\!\int a(\theta)\chi_\theta\, \Delta_v\varphi_\theta\dif\theta.
\end{align*}
Combining these limits proves \eqref{MC-kinetic-solution1}.
\end{proof}
\br
By equation~\eqref{MC-kinetic-solution1}, one sees that in the distributional sense, $\chi_\theta$ satisfies
\begin{align*}
\partial_t\chi_\theta
=a(\theta)\Delta_v\chi_\theta-v\cdot\nabla_x\chi_\theta
+\partial_\theta\pi,
\end{align*}
with the  finite graph measure on every finite time interval
\begin{align*}
\pi(\dif t,\dif z,\dif\theta)
:=
a(f)|\nabla_vf|^2
\delta_{f(t,z)}(\dif\theta)\dif z\dif t.
\end{align*}
Note that the kinetic identity \eqref{MC-kinetic-solution1} holds only for test functions with compact supports in
$\theta\in(0,\infty)$. Identifying $f$ with a renormalized kinetic
solution in the full sense of
Definition~\ref{def-renormalized-kinetic} additionally requires an
extension of the formulation to the boundary level $\theta=0$.
\er

We now establish a priori estimates for regular weak solutions.
\bp\label{prop:regular-L1-Lq}
Let $0\leq f_0\in L^1\cap L^2$, and let $f$ be a weak solution of
\eqref{Aq1} in the sense of Definition~\ref{def-weaksolution}. Suppose that
$a$ is bounded by $\bar a$ and
$\nabla_v f\in L^2_{\mathrm{loc}}(\mR_+\times\mR^{2d})$.
Then for every $q\in[1,2]$,
\begin{align}\label{regular-Lq-estimate}
    \|f(t)\|_{L^q}^q
    +q(q-1)\int_0^t\!\int
       a(f)f^{q-2}|\nabla_vf|^2
    \leq\|f_0\|_{L^q}^q,\quad t\geq0,
\end{align}
where the second term is omitted when $q=1$.
For $q\in(1,2]$, its integrand is understood on $\{f>0\}$
in the displayed form and is defined to be zero on $\{f=0\}$.
\ep
\begin{proof}
(i) We begin with $q=1$.
For $\delta>0$, define
\begin{equation*}
    \beta_\delta(s):=\sqrt{s^2+\delta^2}-\delta,\qquad s\in\mR.
\end{equation*}
Then
\begin{align}\label{As7}
    |\beta_\delta(s)|\leq |s|,\qquad
    |\beta_\delta'(s)|\leq1,\qquad
    \beta_\delta''(s)
    =\frac{\delta^2}{(s^2+\delta^2)^{3/2}}\geq0,
\end{align}
and $\beta_\delta(s)\uparrow|s|$ as $\delta\downarrow0$.
For $R\geq1$, set
\begin{equation*}
    D_R:=B_{R^2}\times B_R,\qquad
    \phi_R(x,v):=\gamma_{R^2}(x)\gamma_R(v),
\end{equation*}
where $\gamma_R$ is defined in \eqref{Ga1}. Thus, for some $C=C(d)>0$,
\begin{equation}\label{regular-cutoff-bounds}
\begin{split}
    &\1_{D_R}\leq\phi_R\leq\1_{D_{2R}},
    \qquad |D_R|\leq C R^{3d},\\
    &|\Delta_v\phi_R|\leq C \1_{D_{2R}}R^{-2},
    \qquad
    |v\cdot\nabla_x\phi_R|\leq C \1_{D_{2R}}R^{-1}.
\end{split}
\end{equation}
Applying \eqref{Au2} with $\beta=\beta_\delta$ gives, for $t>0$,
\begin{align*}
&\int \beta_\delta(f(t))\phi_R
+\int_0^t\!\!\!\int
    a(f)\beta_\delta''(f)|\nabla_vf|^2\phi_R\\
&\quad=\int \beta_\delta(f_0)\phi_R-\int_0^t\!\!\!\int
    a(f)\beta_\delta'(f)(\nabla_vf\cdot\nabla_v\phi_R)
+\int_0^t\!\!\!\int
    \beta_\delta(f)(v\cdot\nabla_x\phi_R).
\end{align*}
Since the second term on the left-hand side is nonnegative, it may be
discarded. Note that by integration by parts in $v$,
\begin{align*}
-\int_0^t\!\!\!\int
    a(f)\beta_\delta'(f)(\nabla_vf\cdot\nabla_v\phi_R)=
\int_0^t\!\!\!\int
   \left(\int_0^f a(\theta)\beta_\delta'(\theta)\dif\theta\right)
   \Delta_v\phi_R.
\end{align*}
For $f\geq0$, the expression in parentheses converges pointwise to
$\Phi_a(f)$ as $\delta\downarrow0$ and is bounded by $\Phi_a(f)$.
Since $\Phi_a(f)\in L^1_{\mathrm{loc}}$, by the dominated convergence theorem, as $\delta\downarrow0$, we have
\begin{align}\label{regular-local-L1}
\int f(t)\phi_R
&\leq\int f_0\phi_R
+\int_0^t\!\!\!\int\Phi_a(f)\Delta_v\phi_R
+\int_0^t\!\!\!\int f(v\cdot\nabla_x\phi_R).
\end{align}
Since $0\leq\Phi_a(f)\leq\bar a f$, by
\eqref{regular-cutoff-bounds}, there is a constant $C_0=C_0(\bar a,d)>0$ such that, for
all $R\geq1$,
\begin{align}\label{regular-JR-recursion}
    J_R(t):=\int_{D_R}f(t)
    \leq \|f_0\|_{L^1}+\frac{C_0}{R}\int_0^tJ_{2R}(s)\dif s.
\end{align}
For fixed $t>0$, since $f\in C_w([0,t]; L^2)$, by the uniform boundedness principle, we have
\begin{equation*}
    M_{2,t}:=\sup_{0\leq s\leq t}\|f(s)\|_{L^2}<\infty.
\end{equation*}
Iterating \eqref{regular-JR-recursion} $n$ times, we get
\begin{align*}
J_R(t)
&\leq
\|f_0\|_{L^1}\sum_{k=0}^{n-1}\frac{(C_0t/R)^k}{k!}
+\left(\frac{C_0}{R}\right)^n
  \int_{0<s_n<\cdots<s_1<t}J_{2^nR}(s_n)
  \dif s_n\cdots\dif s_1\\
&\leq\|f_0\|_{L^1}\sum_{k=0}^{n-1}\frac{(C_0t/R)^k}{k!}
+C_1 M_{2,t}R^{3d/2}
  \frac{(C_0t\,2^{3d/2}/R)^n}{n!},
\end{align*}
where in the second inequality we used
\begin{equation*}
    J_{2^nR}(s)
    \leq |D_{2^nR}|^{1/2}\|f(s)\|_{L^2}
    \leq C_1(2^nR)^{3d/2}M_{2,t}.
\end{equation*}
For fixed $R$, the last term tends to zero as $n\to\infty$. Hence
\begin{equation*}
 \int_{D_R}f(t)=J_R(t)\leq \|f_0\|_{L^1}\e^{C_0t/R}.
\end{equation*}
Letting $R\to\infty$ gives \eqref{regular-Lq-estimate} for $q=1$.

(ii) Fix $q\in(1,2]$ and $\delta\in(0,1)$. Define the nonnegative
$C^2$ function
\begin{equation*}
    \beta_\delta(s)
    :=(|s|+\delta)^q-\delta^q-q\delta^{q-1}|s|,\qquad s\in\mR.
\end{equation*}
For $s\geq0$, we have
\begin{equation*}
    \beta_\delta'(s)
    =q\big((s+\delta)^{q-1}-\delta^{q-1}\big)\geq 0,
    \qquad
    \beta_\delta''(s)=q(q-1)(s+\delta)^{q-2}.
\end{equation*}
This function satisfies \eqref{Au1}, and hence, by \eqref{Au2},
\begin{align}\label{0613:01}
\begin{split}
    &\int\beta_\delta(f(t))\phi_R
    +\int_0^t\!\!\!\int
       \beta_\delta''(f)a(f)|\nabla_vf|^2\phi_R\\
    &\qquad=
    \int\beta_\delta(f_0)\phi_R
    -\int_0^t\!\!\!\int
       \beta_\delta'(f)a(f)
       (\nabla_vf\cdot\nabla_v\phi_R)
    +\int_0^t\!\!\!\int
       (v\cdot\nabla_x\phi_R)\beta_\delta(f).
\end{split}
\end{align}
For $s\geq0$ and $\delta\in(0,1)$, note that
\begin{equation*}
    0\leq\beta_\delta(s)\leq C_q(s+s^2).
\end{equation*}
By (i) and the definition of $M_{2,t}$, we have
\begin{align}\label{regular-beta-integrability}
    \int_0^t\!\!\!\int\beta_\delta(f)
    \leq C_{q,t}\big(\|f_0\|_{L^1}+M_{2,t}^2\big)<\infty.
\end{align}
By integration by parts in $v$, \eqref{regular-cutoff-bounds}, and
\eqref{regular-beta-integrability}, we have
\begin{align*}
&\left|
    \int_0^t\!\!\!\int
    \beta_\delta'(f)a(f)
    (\nabla_vf\cdot\nabla_v\phi_R)
  \right|=
  \left|
    \int_0^t\!\!\!\int
    \left(\int_0^f\beta_\delta'(\theta)a(\theta)\dif\theta\right)
    \Delta_v\phi_R
  \right|\\
&\qquad\leq
  \frac{C\bar a}{R^2}\int_0^t\!\!\!\int\beta_\delta(f)
  \longrightarrow0
  \quad\text{as }R\to\infty.
\end{align*}
Similarly,
\begin{equation*}
\left|
    \int_0^t\!\!\!\int
    \beta_\delta(f)v\cdot\nabla_x\phi_R
\right|
\leq\frac{C}{R}\int_0^t\!\!\!\int\beta_\delta(f)
\longrightarrow0 \quad\text{as }R\to\infty.
\end{equation*}
Letting $R\to\infty$ in \eqref{0613:01} and applying dominated
convergence and Fatou's lemma, we obtain
\begin{align*}
\int\beta_\delta(f(t))
+q(q-1)\int_0^t\!\!\!\int
    a(f)(f+\delta)^{q-2}|\nabla_vf|^2
\leq\int\beta_\delta(f_0).
\end{align*}
Finally, letting $\delta\downarrow 0$, we obtain \eqref{regular-Lq-estimate}.
\end{proof}

Under a suitable growth condition on $\Phi_a$, we can show that weak solutions preserve mass.
\begin{proposition}\label{prop-existence-weak-pde}
Let $0\leq f_0\in L^1\cap L^2$, and let
$f\in\bigcap_{T>0}L^1((0,T)\times\mR^{2d})$ be a nonnegative weak
solution with initial datum $f_0$ in the sense of
Definition~\ref{def-weaksolution}.
Suppose that, for some $\sigma\in(0,1]$ and $\kappa>0$,
\begin{align}\label{mass-growth-condition}
    0\leq\Phi_a(s)\leq\kappa(s^\sigma+s^2),
    \qquad s\geq0,
\end{align}
and that
\begin{align}\label{mass-critical-exponent}
    \sigma\geq1-\tfrac1d\quad\text{if }d\geq2,
    \qquad
    \sigma>0\quad\text{if }d=1.
\end{align}
Then
\begin{align}\label{preservation-mass}
    \|f(t)\|_{L^1}=\|f_0\|_{L^1},
    \qquad t\geq0.
\end{align}
\end{proposition}

\begin{proof}
Let $\phi_R(x,v):=\gamma_R(x)\gamma_R(v)$, where $\gamma_R$ is defined
in \eqref{Ga1}.
Testing the weak formulation \eqref{Weak1} by $\phi_R$ gives
\begin{align}\label{mass-cutoff-identity}
\begin{split}
    \int f(t)\phi_R
    &=\int f_0\phi_R
      +\int_0^t\!\int\Phi_a(f)\Delta_v\phi_R
      +\int_0^t\!\int f\,v\cdot\nabla_x\phi_R.
\end{split}
\end{align}
Set $E_R:=\{|x|\leq2R,\ R\leq|v|\leq2R\}$.
By \eqref{mass-growth-condition}, we have
\begin{align*}
    \int_0^t\!\int\Phi_a(f)|\Delta_v\phi_R|
    &\leq
    \frac{C}{R^2}\int_0^t\!\int_{E_R}(f^\sigma+f^2).
\end{align*}
For the sublinear term, by H\"older's inequality, we have
\begin{align}\label{mass-sublinear-bound}
    \frac1{R^2}\int_0^t\!\int_{E_R}f^\sigma
    \leq
    C_tR^{-2+2d(1-\sigma)}
    \left(\int_0^t\!\int_{E_R}f\right)^\sigma.
\end{align}
Since the power of $R$ is nonpositive under
\eqref{mass-critical-exponent} and
$E_R\subset\{|v|\geq R\}$,
the right-hand side of \eqref{mass-sublinear-bound} tends to zero as $R\to\infty$.
Moreover, as $R\to\infty$, we also have
\begin{align*}
    \frac1{R^2}\int_0^t\!\int_{E_R}f^2
    \leq\frac{t}{R^2}\|f\|_{L^\infty(0,t;L^2)}^2
    \longrightarrow0,
\end{align*}
and
\begin{align*}
    \left|
    \int_0^t\!\int f\,v\cdot\nabla_x\phi_R
    \right|
    \lesssim
    \int_0^t\!\int_{\{|x|\geq R\}}f
    \longrightarrow0.
\end{align*}
Since $0\leq\phi_R\leq1$ and $\phi_R\to1$ pointwise, by Fatou's lemma and
the preceding bounds, we have
\begin{align*}
    \int f(t)\dif z\leq\|f_0\|_{L^1},
\end{align*}
which in turn yields \eqref{preservation-mass} by using the dominated convergence for
\eqref{mass-cutoff-identity}.
\end{proof}

\subsection{Recovery of the weak equation from renormalized kinetic solutions}

We establish the kinetic-level cutoff properties used below.
\begin{proposition}\label{prop:kinetic-to-weak}
Let $(f,\pi)$ be a mass-preserving renormalized kinetic solution to
\eqref{Aq1}. Then, for every $T>0$,
\begin{align}\label{small-kinetic-level}
    \lim_{\delta\downarrow0}
    \pi\bigl([0,T]\times\mR^{2d}\times[\delta/2,\delta]\bigr)/\delta=0,
\end{align}
and, with $\zeta_\delta$ defined by \eqref{kinetic-cutoff},
\begin{align}\label{cutoff-measure-vanishing}
    \lim_{\delta\downarrow0}
    \int_{[0,T]\times\mR^{2d}\times \mR_+}
    |\zeta_\delta'(\theta)|\dif\pi=0.
\end{align}
If, in addition,
$\Phi_a(f)\in L^1_{\mathrm{loc}}(\mR_+\times\mR^{2d})$, then $f$ is a
weak solution in the sense of Definition~\ref{def-weaksolution}. In
particular, this additional integrability holds under
Assumption~\ref{A3}.
\end{proposition}

\begin{proof}
Let $\eta\in C_c^\infty(\mR_+;[0,1])$ and set
\begin{align*}
    \phi_R(x,v):=\gamma_{R^2}(x)\gamma_R(v),
\end{align*}
where $\gamma_R$ is defined in \eqref{Ga1}. Using
$\eta(t)\phi_R(z)\zeta_\delta(\theta)$ as a test function in
\eqref{kinetic-formulation-distributional} yields
\begin{align}\label{Ar2}
\begin{split}
    &\int_{\mR_+\times\mR^{2d}\times \mR_+}
    \chi\zeta_\delta\bigl(\eta'\phi_R+\eta a(\theta)\Delta_v\phi_R
    +\eta v\cdot\nabla_x\phi_R\bigr)\dif\theta\dif z\dif t\\
    &\quad-
    \int_{\mR_+\times\mR^{2d}\times\mR_+}
    \eta\phi_R\zeta_\delta'\,\dif\pi
    +\int_{\mR^{2d}}\!\int_{\mR_+}
    \eta(0)\chi_0\phi_R\zeta_\delta
    \dif\theta\dif z=0.
\end{split}
\end{align}
For fixed $\delta>0$, let $\zeta_\delta$ be defined by \eqref{kinetic-cutoff} and define
\begin{align}\label{CB0}
    \cC_\delta(s):=\int_0^s\zeta_\delta(\theta)\dif\theta,
    \qquad
    \cB_\delta(s):=\int_0^s a(\theta)\zeta_\delta(\theta)\dif\theta.
\end{align}
Since $0\leq\zeta_\delta\leq1$ and its support is a compact subset
of $(0,\infty)$, there is a constant $c_\delta<\infty$ such that
\begin{align*}
    0\leq\cC_\delta(s)\leq s,
    \qquad
    0\leq\cB_\delta(s)\leq c_\delta s,
    \qquad s\geq0.
\end{align*}
Moreover, as $R\to\infty$,
\begin{align*}
    \left|\int \eta\cB_\delta(f)\Delta_v\phi_R\right|
    &\leq \frac{C_{\delta,\eta}}{R^2}
    \int_{\operatorname{supp}\eta}\!\int f
    \longrightarrow0,\\
    \left|\int \eta\cC_\delta(f)
    v\cdot\nabla_x\phi_R\right|
    &\leq \frac{C_\eta}{R}
    \int_{\operatorname{supp}\eta}\!\int f
    \longrightarrow0.
\end{align*}
Letting $R\to\infty$ in \eqref{Ar2} and applying dominated convergence, we obtain
\begin{align}\label{BV-Cdelta-distributional}
    &\int_{\mR_+}\!\int_{\mR^{2d}}
    \cC_\delta(f(t,z))\eta'(t)\dif z\dif t
    -\int_{\mR_+\times\mR^{2d}\times\mR_+}
    \eta(t)\zeta_\delta'(\theta)\dif\pi
    +\eta(0)\int_{\mR^{2d}}\cC_\delta(f_0(z))\dif z=0.
\end{align}
We pass from \eqref{BV-Cdelta-distributional} to a fixed time. Let
$\varrho\in C_c^\infty((0,1);[0,\infty))$ with
$\int_\mR\varrho=1$, and set
\begin{align*}
    \varrho_\eps(r):=\eps^{-1}\varrho(r/\eps),
    \qquad
    \eta_{\eps,S}(t):=\int_{t-S}^\infty\varrho_\eps(r)\dif r.
\end{align*}
Then $\eta_{\eps,S}=1$ on $[0,S]$, it vanishes on
$[S+\eps,\infty)$, and
$\eta_{\eps,S}'(t)=-\varrho_\eps(t-S)$. For fixed $\delta$, put
\begin{align*}
    G_\delta(t):=
    \int_{\mR^{2d}}\cC_\delta(f(t,z))\dif z.
\end{align*}
Since $0\leq G_\delta(t)\leq\|f(t)\|_{L^1}$, we have
$G_\delta\in L^\infty_{\mathrm{loc}}(\mR_+)$. If $S>0$ is a right
Lebesgue point of $G_\delta$, then by \eqref{BV-Cdelta-distributional} with $\eta=\eta_{\eps,S}$
 and letting $\eps\downarrow0$,
by the right-sided Lebesgue differentiation theorem and the dominated convergence theorem
with respect to the finite measure $\pi$, we get
\begin{align}\label{BV-Cdelta}
    &\int_{\mR^{2d}}\cC_\delta(f(S,z))\dif z
    +\int_{[0,S]\times\mR^{2d}\times\mR_+}
    \zeta_\delta'(\theta)\dif\pi
    -\int_{\mR^{2d}}\cC_\delta(f_0(z))\dif z=0.
\end{align}
Set $\delta_n:=2^{-n}$ for $n\geq3$. Taking the countable intersection
of the corresponding Lebesgue-point sets, we obtain a set
$\mathcal L\subset(0,\infty)$ of full measure such that every
$S\in\mathcal L$ is a right Lebesgue point of every function
$G_{\delta_n}$. For $S\in\mathcal L$, apply
\eqref{BV-Cdelta} with $\delta=\delta_n$. Since
$\cC_{\delta_n}(s)\to s$ and $0\leq\cC_{\delta_n}(s)\leq s$,
by dominated convergence and mass preservation, we have
\begin{align}\label{weighted-signed-vanishing}
    \int_{[0,S]\times\mR^{2d}\times\mR_+}
    \zeta_{\delta_n}'(\theta)\dif\pi
    &\longrightarrow
    \int_{\mR^{2d}}f_0(z)\dif z
    -\int_{\mR^{2d}}f(S,z)\dif z=0.
\end{align}
On the other hand, by the last inequality in \eqref{kinetic-cutoff}, we have
\begin{align*}
    &\frac{1}{\delta_n}
    \pi\bigl([0,S]\times\mR^{2d}
    \times[\delta_n/2,\delta_n]\bigr)\leq
    \int_{[0,S]\times\mR^{2d}\times\mR_+}
    \zeta_{\delta_n}'(\theta)\dif\pi
    +2\pi\bigl([0,S]\times\mR^{2d}
    \times[\delta_n^{-1},\delta_n^{-1}+1]\bigr).
\end{align*}
The first term tends to zero by \eqref{weighted-signed-vanishing}, and
the second tends to zero because
$\pi([0,S]\times\mR^{2d}\times\mR_+)<\infty$. 
Hence, for every $S\in\mathcal L$,
\begin{align*}
    \frac{1}{\delta_n}
    \pi\bigl([0,S]\times\mR^{2d}
    \times[\delta_n/2,\delta_n]\bigr)\stackrel{n\to\infty}{\longrightarrow}0.
\end{align*}
For general $\delta\in(\delta_{n+1},\delta_n]$, noting that
\begin{align*}
    &\frac{\pi\bigl([0,S]\times\mR^{2d}\times[\delta/2,\delta]\bigr)}{\delta}
    \leq
    \frac{2\pi\bigl([0,S]\times\mR^{2d}
    \times[\delta_n/2,\delta_n]\bigr)}{\delta_n}
    +\frac{\pi\bigl([0,S]\times\mR^{2d}
    \times[\delta_{n+1}/2,\delta_{n+1}]\bigr)}{\delta_{n+1}},
\end{align*}
the preceding dyadic limit proves \eqref{small-kinetic-level} for every
$S\in\mathcal L$.
Given any $T>0$, choose
$S\in\mathcal L$ with $S>T$. The same limit for $T$ then follows from
\begin{align*}
\pi([0,T]\times\mR^{2d}\times[\delta/2,\delta])
\leq
\pi([0,S]\times\mR^{2d}\times[\delta/2,\delta]).
\end{align*}
Furthermore, by \eqref{kinetic-cutoff}, we have
\begin{align*}
    &\int_{[0,T]\times\mR^{2d}\times\mR_+}
    |\zeta_\delta'(\theta)|\dif\pi\leq
    \frac{2\pi\bigl([0,T]\times\mR^{2d}\times[\delta/4,2\delta]\bigr)}{\delta}
    +2\pi\bigl([0,T]\times\mR^{2d}
    \times[\delta^{-1},\infty)\bigr).
\end{align*}
The first term tends to zero by \eqref{small-kinetic-level}, after
decomposing $[\delta/4,2\delta]$ into the three intervals
$[\varepsilon/2,\varepsilon]$ with
$\varepsilon=\delta/2,\delta,2\delta$.
The second term tends to zero by the finiteness of $\pi$. 
This proves \eqref{cutoff-measure-vanishing}.

Assume finally that
$\Phi_a(f)\in L^1_{\mathrm{loc}}(\mR_+\times\mR^{2d})$.
Fix $\varphi\in C_c^\infty(\mR^{2d})$ and
$\eta\in C_c^\infty(\mR_+)$. Using $\eta(t)\varphi(z)\zeta_\delta(\theta)$
as a test function in \eqref{kinetic-formulation-distributional} gives
\begin{align}\label{weak-recovery-cutoff}
\begin{split}
&    \int_0^\infty\!\!\!\!\int_{\mR^{2d}}
      \cC_\delta(f)\varphi\eta'\dif z\dif t
    +
      \int_0^\infty\!\!\!\!\int_{\mR^{2d}}
      \eta\left[
      \cB_\delta(f)\Delta_v\varphi
      +\cC_\delta(f)v\cdot\nabla_x\varphi
      \right]\dif z\dif t\\
    &\quad-
      \int_{\mR_+\times\mR^{2d}\times\mR_+}
      \eta(t)\varphi(z)\zeta_\delta'(\theta)\dif\pi
      +\eta(0)\int_{\mR^{2d}}
      \cC_\delta(f_0)\varphi\dif z=0,
\end{split}
\end{align}
where $\cC_\delta$ and $\cB_\delta$ are defined by \eqref{CB0}.
For $s\geq0$, $\cC_\delta(s)\to s$ and
$\cB_\delta(s)\to\Phi_a(s)$ as $\delta\downarrow0$, with
\begin{align*}
    0\leq\cC_\delta(s)\leq s,
    \qquad
    0\leq\cB_\delta(s)\leq\Phi_a(s).
\end{align*}
Thus, letting $\delta\downarrow0$ and applying
\eqref{cutoff-measure-vanishing} together with dominated convergence,
we obtain
\begin{align}\label{weak-recovery-distributional}
\begin{split}
\int_0^\infty\!\!\!\!\int_{\mR^{2d}}
      f(t)\varphi\eta'(t)\dif z\dif t+
      \int_0^\infty\!\!\!\!\int_{\mR^{2d}}
      \eta(t)\left[
      \Phi_a(f)\Delta_v\varphi
      +f\,v\cdot\nabla_x\varphi
      \right]\dif z\dif t
      +\eta(0)\int_{\mR^{2d}}f_0\varphi\dif z=0.
\end{split}
\end{align}
Set
\begin{align*}
    F_\varphi(t)
    :=
    \int_{\mR^{2d}}f(t,z)\varphi(z)\dif z,\quad
    G_\varphi(t)
    :=
    \int_{\mR^{2d}}
      \left[
      \Phi_a(f(t,z))\Delta_v\varphi(z)
      +f(t,z)v\cdot\nabla_x\varphi(z)
      \right]\dif z.
\end{align*}
Since $f\in C_w(\mR_+;L^2)$ and $\Phi_a(f)\in L^1_{\mathrm{loc}}(\mR_+\times\mR^{2d})$, 
we have $F_\varphi\in C(\mR_+)$ and
$G_\varphi\in L^1_{\mathrm{loc}}(\mR_+)$.
Taking $\eta\in C_c^\infty((0,\infty))$ in
\eqref{weak-recovery-distributional} shows that
\begin{align*}
    F_\varphi'=G_\varphi
    \qquad\text{in }\mathcal D'((0,\infty)).
\end{align*}
Thus, for some constant $c_\varphi$,
\begin{align*}
    F_\varphi(t)=c_\varphi+\int_0^tG_\varphi(s)\dif s,
    \qquad t\geq0.
\end{align*}
Substituting this identity into
\eqref{weak-recovery-distributional} and allowing
$\eta(0)\neq0$ yields
\begin{align*}
    c_\varphi=\int_{\mR^{2d}}f_0(z)\varphi(z)\dif z.
\end{align*}
Consequently, for every $t\geq0$,
\begin{align*}
    \int_{\mR^{2d}}f(t)\varphi\dif z
 =\int_{\mR^{2d}}f_0\varphi\dif z+
      \int_0^t\!\!\int_{\mR^{2d}}
      \left[
      \Phi_a(f)\Delta_v\varphi
      +f\,v\cdot\nabla_x\varphi
      \right]\dif z\dif s.
\end{align*}
Hence $f$ is a weak solution.
\end{proof}

\section{Uniformly elliptic diffusion coefficients}\label{sec:elliptic}

In this section we assume that $a$ is defined on $\mR$ as a positive
$C^1$ function satisfying
\begin{align}\label{uniform-ellipticity}
    0<\underline a\leq a(s)\leq\overline a,
    \qquad
    |a'(s)|\leq L_a,
    \qquad s\in\mR.
\end{align}

\begin{theorem}\label{Th41}
Suppose that \eqref{uniform-ellipticity} holds. For every nonnegative
$f_0\in L^1(\mR^{2d})\cap L^2(\mR^{2d})$, equation \eqref{Aq1}
admits a nonnegative weak solution $f$ such that, for every
$q\in[1,2]$ and $t\geq0$,
\begin{align}\label{Llambda-es}
    \|f(t)\|_{L^q}^q
    +q(q-1)\int_0^t\!\int
       a(f)f^{q-2}|\nabla_vf|^2
    \leq\|f_0\|_{L^q}^q,
\end{align}
where, for $q\in(1,2]$, the dissipation density is evaluated on
$\{f>0\}$ and defined to be zero on $\{f=0\}$; for $q=1$, the
dissipation term is omitted. In addition,
\begin{align}\label{uniform-elliptic-mass}
    \|f(t)\|_{L^1}=\|f_0\|_{L^1},
    \qquad t\geq0.
\end{align}
Moreover, the solution is a renormalized kinetic solution with the exact kinetic
measure
\begin{align}\label{exact-kinetic-measure}
    \pi(\dif t,\dif z,\dif\theta)
    =|\nabla_vH_a(f(t,z))|^2
      \delta_{f(t,z)}(\dif\theta)\dif z\dif t.
\end{align}
\end{theorem}

\begin{proof}
For $n\geq1$, consider the uniformly parabolic approximation
\begin{align}\label{elliptic-space-approximation}
    \partial_tf_n
    =\operatorname{div}_v(a(f_n)\nabla_vf_n)
      +n^{-1}\Delta_xf_n-v\cdot\nabla_xf_n,
    \qquad f_n(0)=f_0.
\end{align}
By the standard theory of PDE, there is a weak solution $f_n$ to 
\eqref{elliptic-space-approximation} satisfying the energy estimates
\begin{align}\label{elliptic-approx-energy}
\begin{split}
    \|f_n(t)\|_{L^2}^2
    &+2\int_0^t\!\int a(f_n)|\nabla_vf_n|^2
      +\frac2n\int_0^t\!\int|\nabla_xf_n|^2
    \leq\|f_0\|_{L^2}^2.
\end{split}
\end{align}
Moreover, by the maximum principle, $f_n\geq0$.

To obtain compactness of $f_n$ in the full phase variable, let $P_t$ be the
semigroup generated by $\Delta_v-v\cdot\nabla_x$, and let $G_t^n$ be the
heat semigroup generated by $n^{-1}\Delta_x$. Since these semigroups
commute, by Duhamel's formula, we have
\begin{align*}
    f_n(t)
    &=P_tG_t^nf_0
      +\int_0^tP_{t-s}G_{t-s}^n\operatorname{div}_v
       \bigl((a(f_n)-1)\nabla_vf_n\bigr)\dif s.
\end{align*}
Fix $\alpha\in(0,1)$. By the kinetic semigroup estimate
\cite[Lemma 2.14]{HRZ25}, the $L^2$ contraction of $G_t^n$, and
\eqref{elliptic-approx-energy}, we have
\begin{align*}
\|f_{n}(t)\|_{\bB_{2;\mathbf a}^{\alpha}}
&\lesssim t^{-\alpha/2}\|\mathfix{G^n_t}f_0\|_{L^2} 
+ \int_0^t (t-s)^{-(\alpha+1)/2}\|\mathfix{G^n_{t-s}}\left((a(f_n)-1) \nabla_v f_{n} \right)\|_{L^2} \dif s\\
&\leq t^{-\alpha/2}\|f_0\|_{L^2} 
+ \mathfix{(\overline a+1)}\int_0^t (t-s)^{-(\alpha+1)/2}\|\nabla_v f_{n}\|_{L^2} \dif s,
\end{align*}
which in turn yields by integrating both sides with respect to $t$ and \eqref{elliptic-approx-energy},
\begin{align*}
    \sup_{n\geq1}
    \|f_n\|_{L^2(0,T;\bB_{2;\mathbf a}^{\alpha})}
    \leq C_T\|f_0\|_{L^2}.
\end{align*}
For every $\varphi\in C_c^\infty(\mR^{2d})$, by equation
\eqref{elliptic-space-approximation}, the boundedness of $a$, and
\eqref{elliptic-approx-energy}, we further have
\begin{align*}
    \sup_{n\geq1}
    \|\partial_t(f_n\varphi)\|_{L^2(0,T;\bB_{2;\mathbf a}^{-6})}
    \leq C_{T,\varphi}\|f_0\|_{L^2}.
\end{align*}
Since $\bB_{2;\mathbf a}^{\alpha}$ is compactly embedded into $L^2$ on bounded
sets, by the Aubin--Lions lemma, after passing to a subsequence
(not relabeled), there is a function
$f\in L^\infty_{\rm loc}(\mR_+;L^2)$ with
$\nabla_vf\in L^2([0,T]\times\mR^{2d};\mR^d)$
such that
\begin{align}\label{elliptic-strong-convergence}
    f_n\longrightarrow f
    \quad\text{almost everywhere and strongly in }L^2(0,T;L^2_{\mathrm{loc}})
\end{align}
and 
\begin{align*}
\nabla_vf_n\rightharpoonup \nabla_vf
\quad\text{weakly in }L^2([0,T]\times\mR^{2d}).
\end{align*}
Since $\Phi_a$ is globally Lipschitz under \eqref{uniform-ellipticity}, we have
\begin{align*}
    \Phi_a(f_n)\longrightarrow\Phi_a(f)
    \quad\text{strongly in }L^2(0,T;L^2_{\mathrm{loc}}).
\end{align*}
Passing to the limit in the weak formulation of
\eqref{elliptic-space-approximation} (the term
$n^{-1}\int f_n\Delta_x\varphi$ tends to zero by the uniform
$L^2$ bound) gives
\begin{align*}
    \int f(t)\varphi
    &=\int f_0\varphi
      +\int_0^t\!\int\Phi_a(f)\Delta_v\varphi\dif z\dif s
      +\int_0^t\!\int f\,v\cdot\nabla_x\varphi\dif z\dif s.
\end{align*}
Since $\|f_n(t)\|_{L^2}\leq\|f_0\|_{L^2}$ and $f_n(t)\geq 0$,
by \eqref{elliptic-strong-convergence} and a standard weak-continuity argument, we have
$f\in C_w([0,T];L^2)$, and $f\geq0$.
Moreover, since $a$ is bounded from above and below, by Propositions~\ref{prop:regular-L1-Lq}
and \ref{prop-existence-weak-pde}, we also have \eqref{Llambda-es} and
\eqref{uniform-elliptic-mass}.

Proposition~\ref{prop:regular-renormalization}(ii) gives the kinetic
identity for test function with supports compactly contained in $(0,\infty)$, with the graph measure \eqref{exact-kinetic-measure}.
By \eqref{Llambda-es} with $q=2$, we have
\begin{align*}
    \pi([0,T]\times\mR^{2d}\times\mR_+)
    \leq\tfrac12\|f_0\|_{L^2}^2.
\end{align*}
It remains to include the boundary level $\theta=0$. Let
$\phi_R(x,v)=\gamma_{R^2}(x)\gamma_R(v)$ and let $\cC_\delta$ be as in
\eqref{CB0}. Applying the fixed-time local kinetic identity with
$\phi_R(z)\zeta_\delta(\theta)$ and then letting $R\to\infty$, we get
\begin{align*}
    \int\cC_\delta(f(t))\dif z
    +\int_{[0,t]\times\mR^{2d}\times\mR_+}
       \zeta_\delta'(\theta)\dif\pi
    =\int\cC_\delta(f_0)\dif z.
\end{align*}
Indeed, the two phase-space boundary terms vanish by the bounds in
\eqref{regular-cutoff-bounds}, the boundedness of $a$, and the global
$L^1$ estimate. Using the mass conservation and the dominated convergence theorem, as in the
proof of Proposition~\ref{prop:kinetic-to-weak}, we have
\begin{align*}
    \lim_{\delta\downarrow0}
    \int_{[0,T]\times\mR^{2d}\times\mR_+}
       |\zeta_\delta'(\theta)|\dif\pi=0.
\end{align*}
For an arbitrary test function in the class of
Definition~\ref{def-renormalized-kinetic}, we may now use
$\zeta_\delta\varphi$ in the local kinetic identity and let
$\delta\downarrow0$. The graph measure does not charge
$\{\theta=0\}$: by locality of Sobolev gradients,
$\nabla_vH_a(f)=0$ almost everywhere on $\{f=0\}$. This proves the
full identity \eqref{kinetic-formulation-distributional}. Finally,
\begin{align*}
    \int h(t,z,f)|\nabla_vH_a(f)|^2\dif z\dif t=\int h\dif\pi
\end{align*}
for every nonnegative $h\in C_c$, so \eqref{controlKM} holds with
equality. Thus $f$ is a renormalized kinetic solution with the exact
measure \eqref{exact-kinetic-measure}.
\end{proof}

In the remainder of this section, we use the semigroup notation of
Section~\ref{sec:semigroup} with $a_\theta:=a(\theta)$. We shall use the
following Duhamel representation; its proof is given in
Appendix~\ref{appendix:Duhamel}.

\begin{lemma}\label{Le34}
Let $f$ be the solution in Theorem~\ref{Th41}, and let $\pi$ be given by
\eqref{exact-kinetic-measure}. Then, in $\cS'(\mR^{2d})$ and for every
$t\geq0$,
\begin{align}\label{Duhamel-formula}
\begin{split}
    f(t)
    &=\int_0^\infty P_t^\theta\chi_0(\theta)\dif\theta-
      \int_0^t\!\int_0^\infty
      \partial_\theta(\Gamma_{t-s}p_{t-s}^\theta)
      *\Gamma_{t-s}\pi(s,\cdot,\dif\theta)\dif s,
\end{split}
\end{align}
where $\chi_0(z,\theta)=\1_{\{0<\theta<f_0(z)\}}$.
\end{lemma}

We now establish the uniform Besov estimate used in the compactness argument.

\begin{proposition}\label{uniform-besov-2}
Suppose that \eqref{uniform-ellipticity} and
\eqref{As15}--\eqref{As16} hold. Let $f$ be the solution in Theorem~\ref{Th41}. For every $p\in[1,\tfrac{2d}{2d-\ell})$ and $T>0$,
we have
\begin{align}\label{besov-es}
    \|f\|_{L^1(0,T;\bB_{p;\mathbf a}^{\alpha})}
    \leq
    C_T\bigl(\|f_0\|_{L^1}+\|f_0\|_{L^2}^2\bigr),
\end{align}
where $\alpha:=2\ell-4d(1-\tfrac1p)$ and $C_T$ depends only on
$d,p,\ell,T,q_0,q_1,\kappa_0,\kappa_1$.
\end{proposition}

\begin{proof}
For $t>0$, write
\eqref{Duhamel-formula} as $f(t)=\sI_1(t)+\sI_2(t)$. For $j\geq0$, by
Lemma~\ref{lem-kineticsemigroup-es} with  $k=0$, Minkowski's inequality,
and \eqref{As16}, we have
\begin{align}\label{initial-Besov-bound}
\begin{split}
    \|\cR_j^{\mathbf a}\sI_1(t)\|_{L^p}
    &\leq
    C2^{[4d(1-1/p)-2\ell]j}t^{-\ell}
    \int_{\mR^{2d}}\!\int_0^{f_0(z)}
      \bigl(a(\theta)^{-\ell}\vee
            a(\theta)^{-\ell/3}\bigr)\dif\theta\dif z\\
    &\leq
    C2^{[4d(1-1/p)-2\ell]j}t^{-\ell}
    \bigl(\|f_0\|_{L^1}+\|f_0\|_{L^2}^2\bigr).
\end{split}
\end{align}
For the kinetic-measure term,
by Lemma~\ref{lem-kineticsemigroup-es} with $k=1$, together with
\eqref{As15}, we have
\begin{align*}
    \|\cR_j^{\mathbf a}
      \partial_\theta(\Gamma_{t-s}p_{t-s}^\theta)\|_{L^p}
    \leq
    C2^{[4d(1-1/p)-2\ell]j}(t-s)^{-\ell}
    (\theta^{q_0-2}+\theta^{q_1-2}).
\end{align*}
Thus, using Minkowski's inequality and \eqref{exact-kinetic-measure}, we obtain
\begin{align}\label{measure-Besov-bound}
\begin{split}
    \|\cR_j^{\mathbf a}\sI_2(t)\|_{L^p}
    &\leq
    C2^{[4d(1-1/p)-2\ell]j}\int_0^t(t-s)^{-\ell}
       \int(f^{q_0-2}+f^{q_1-2})
       a(f)|\nabla_vf|^2\dif z\dif s.
\end{split}
\end{align}
As in Theorem~\ref{Th41}, the integrand in
\eqref{measure-Besov-bound} is set to zero on $\{f=0\}$; the exact graph
measure assigns no mass there.
The low-frequency block is controlled directly by Young's inequality and
mass conservation:
\begin{align*}
    \|\mathcal{R}^{\mathbf a}_{-1}f(t)\|_{L^p}
    \leq
    \|\check\phi_{-1}^{\mathbf a}\|_{L^p}\|f(t)\|_{L^1}
    =
    \|\check\phi_{-1}^{\mathbf a}\|_{L^p}\|f_0\|_{L^1}.
\end{align*}
Multiplying
\eqref{initial-Besov-bound} and \eqref{measure-Besov-bound} by
$2^{\alpha j}$, taking the supremum in $j$, and integrating in time, we
obtain
\begin{align*}
    \int_0^T\|f(t)\|_{\bB_{p;\mathbf a}^{\alpha}}\dif t
    &\leq C_T\bigl(\|f_0\|_{L^1}+\|f_0\|_{L^2}^2\bigr)+C_T\sum_{i=0,1}
       \int_0^T\!\int f^{q_i-2}a(f)|\nabla_vf|^2\dif z\dif s\\
       &\!\!\!\stackrel{\eqref{Llambda-es}}{\leq} C_T\bigl(\|f_0\|_{L^1}+\|f_0\|_{L^2}^2\bigr)+C_T\sum_{i=0,1}
       \|f_0\|_{L^{q_i}}^{q_i}\\
       &\leq C_T\bigl(\|f_0\|_{L^1}+\|f_0\|_{L^2}^2\bigr),
\end{align*}
where the last inequality follows from
$
    \|f_0\|_{L^{q_i}}^{q_i}
    \leq\|f_0\|_{L^1}+\|f_0\|_{L^2}^2.
$
This proves \eqref{besov-es}.
\end{proof}
\section{Construction for general diffusion coefficients}\label{sec:existence}

Let $(a_\eps)_{\eps\in(0,1)}$ be
the family constructed in Lemma~\ref{approx-Psi}. Each $a_\eps$ is
positive, bounded, and belongs to $C_b^1(\mR)$. Moreover, the constants
in \eqref{As15}--\eqref{As17} are uniform in $\eps$.
Let $f_\eps$ be the solution of
\begin{align}\label{approximating-equation}
    \partial_tf_\eps
    =\operatorname{div}_v\bigl(a_\eps(f_\eps)\nabla_vf_\eps\bigr)
      -v\cdot\nabla_xf_\eps,
    \qquad f_\eps(0)=f_0,
\end{align}
constructed in Theorem~\ref{Th41}. By
\eqref{Llambda-es} and
Proposition~\ref{uniform-besov-2}, for every $T>0$,
\begin{align}\label{uniform-main-estimates}
\begin{split}
    &\sup_{\eps\in(0,1)}\left(
      \|f_\eps\|_{L^\infty(0,T;L^1\cap L^2)}+
      \int_{[0,T]\times\mR^{2d}}
      |\nabla_vH_{a_\eps}(f_\eps)|^2\right)<\infty,\\
    &\qquad\sup_{\eps\in(0,1)}
      \|f_\eps\|_{L^1(0,T;\bB_{p;\mathbf a}^{\alpha})}
      \leq C_T\bigl(\|f_0\|_{L^1}+\|f_0\|_{L^2}^2\bigr),
\end{split}
\end{align}
where $p\in[1,\tfrac{2d}{2d-\ell})$ and $\alpha:=2\ell-4d(1-\tfrac1p)$.

\subsection{Compactness}

Fix $R>0$ and choose $\varphi,\phi_R\in C_c^\infty(\mR^{2d})$ such that
$\varphi=1$ on $B_R$ and $\phi_R=1$ on a neighborhood of
$\operatorname{supp}\varphi$.
Since
$\Phi_{a_\eps}(f_\eps)\leq C(1+f^2_\eps)$ and $f_\eps$ is uniformly
bounded in $L^\infty(0,T;L^1\cap L^2)$, by equation
\eqref{approximating-equation}, we have
\begin{align}\label{general-time-compactness}
    \|\partial_tf_{\eps} \varphi\|_{L^\infty([0,T]; \bB_{1;\mathbf a}^{-3})}
    &\leq\|\Delta_v(\Phi_{a_\eps}(f_\eps))\varphi\|_{L^\infty([0,T]; \bB_{1;\mathbf a}^{-3})}
    + \|(v \cdot \nabla_x f_\eps) \varphi\|_{L^\infty([0,T]; \bB_{1;\mathbf a}^{-3})}\no\\
    &=\|\Delta_v(\phi_R\Phi_{a_\eps}(f_\eps))\varphi\|_{L^\infty([0,T]; \bB_{1;\mathbf a}^{-3})}
    + \|(v \cdot \nabla_x (\phi_R f_\eps)) \varphi\|_{L^\infty([0,T]; \bB_{1;\mathbf a}^{-3})}\no\\
    &\lesssim\|\phi_R\Phi_{a_\eps}(f_\eps)\|_{L^\infty([0,T];L^1)}+\|\phi_Rf_\eps\|_{L^\infty([0,T]; L^1)}\no\\
    &\lesssim\|\phi_R(1+|f_\eps|^2)\|_{L^\infty([0,T];L^1)}+\|\phi_Rf_\eps\|_{L^\infty([0,T]; L^1)}
    \lesssim 1+\|f_0\|^2_{L^2}.
\end{align}
Since  $\bB_{1;\mathbf a}^{2\ell}$ is locally embedded in $L^1$, by the Aubin--Lions lemma,
together with \eqref{uniform-main-estimates} and
\eqref{general-time-compactness},
there is a subsequence (not relabeled) and a nonnegative function $f$ such that
\begin{align}\label{main-strong-L1}
    f_\eps\longrightarrow f
    \quad\text{almost everywhere and strongly in }
    L^1((0,T)\times B_R)
    \quad\text{for every }T,R>0.
\end{align}
The uniform $L^2$ bound in \eqref{uniform-main-estimates} then implies, by
interpolation,
\begin{align}\label{main-strong-Lq}
    f_\eps\longrightarrow f
    \quad\text{strongly in }L^q_{\mathrm{loc}}
    ((0,\infty)\times\mR^{2d})
    \quad\text{for every }q\in[1,2).
\end{align}

To identify the nonlinear term, fix $T,R>0$ and set
$K:=(0,T)\times B_R$. Since $\sigma_1<2$ and
$0\leq\sigma_0\leq1$, the $L^2(K)$ bound implies that
$f_\eps^{\sigma_0}+f_\eps^{\sigma_1}$ is bounded in $L^r(K)$ for some
$r>1$, with the convention $s^0=1$. Hence,
by \eqref{primitive-comparison}, \eqref{As17}, and de la Vall\'ee Poussin's
criterion, $(\Phi_{a_\eps}(f_\eps))_\eps$ is uniformly
integrable on $K$. Thus, by Vitali's theorem, we get
\begin{align}\label{nonlinear-strong-convergence}
    \Phi_{a_\eps}(f_\eps)
    \longrightarrow\Phi_a(f)
    \quad\text{strongly in }L^1((0,T)\times B_R)
    \quad\text{for every }T,R>0.
\end{align}
Similarly, for $q\in(1,2]$, define
$H_{a_\eps,q}$ as in Theorem~\ref{thm:main} with $a$ replaced by
$a_\eps$.  For every $R>0$,
\begin{align*}
\sup_{s\in[0,R]}
    |H_{a_\eps,q}(s)-H_{a,q}(s)|
    &\leq
    \left(\int_0^R s^{q-2}\dif s\right)^{1/2}
    \left(\int_0^R
    |\sqrt{a_\eps(s)}-\sqrt{a(s)}|^2\dif s\right)^{1/2}
    \longrightarrow0,
\end{align*}
where we used $|\sqrt{x}-\sqrt y|^2\leq|x-y|$ and
\eqref{approx-local-L1}.  Moreover, noting that 
\begin{align*}
    H_{a_\eps,q}(s)^2
    \leq \frac{s^{q-1}}{q-1}\Phi_{a_\eps}(s)
\leq \frac{1}{q-1}\left[
      \kappa_2\bigl(s^{\sigma_0+q-1}+s^{\sigma_1+q-1}\bigr)
      +\lambda_\eps s^q+2\eps s^{q-1}
    \right],
\end{align*}
where we used \eqref{primitive-comparison} and \eqref{As17}.
Since $(\lambda_\eps)$ is bounded and
$(\sigma_1+q-1)/2<2$, this estimate makes
$H_{a_\eps,q}(f_\eps)$ uniformly integrable on each finite cylinder.
Together with the preceding local uniform convergence,
\eqref{main-strong-L1}, and Vitali's theorem, it implies
\begin{align}\label{H-strong-convergence}
    H_{a_\eps,q}(f_\eps)
    \longrightarrow H_{a,q}(f)
    \quad\text{strongly in }L^1((0,T)\times B_R)
    \quad\text{for every }T,R>0.
\end{align}

\subsection{Proof of Theorem~\ref{thm:main}}
\begin{proof}
For each $\varphi\in C^\infty_c(\mR^{2d})$, the weak formulation of
\eqref{approximating-equation} gives
\begin{align*}
\int f_\eps(t)\varphi
&=\int f_\eps(0)\varphi
  +\int_0^t\!\!\int \Phi_{a_\eps}(f_\eps)\Delta_v\varphi
  +\int_0^t\!\!\int f_\eps\,v\cdot\nabla_x\varphi.
\end{align*}
Letting $\eps\to0$ and using \eqref{main-strong-L1} and
\eqref{nonlinear-strong-convergence} yields \eqref{Weak1} for almost
every $t\geq0$.
Combined with the weak-* bound in
$L^\infty(0,T;L^2)$, this yields a representative
$f\in C_w([0,T];L^2)$ satisfying $f(0)=f_0$ and \eqref{Weak1} for every
$t\geq 0$. Thus $f$ is a weak solution.
Moreover, by \eqref{uniform-main-estimates} and \eqref{H-strong-convergence}, if necessary, 
through selecting a subsequence, we also have
\begin{align*}
    \nabla_vH_{a_\eps,q}(f_\eps)\rightharpoonup G=\nabla_vH_{a,q}(f)
    \quad\text{weakly in }L^2([0,T]\times\mR^{2d}).
\end{align*}
Lower semicontinuity therefore
gives \eqref{main-Lq} and \eqref{main-besov}.

For each $\eps$, let
$\chi_\eps(t,z,\theta):=\1_{\{0<\theta<f_\eps(t,z)\}}$,
$\chi(t,z,\theta):=\1_{\{0<\theta<f(t,z)\}}$, and
\begin{align}\label{approximate-kinetic-measure}
    \pi_\eps(\dif t,\dif z,\dif\theta)
    :=|\nabla_vH_{a_\eps}(f_\eps)|^2
      \delta_{f_\eps(t,z)}(\dif\theta)\dif z\dif t.
\end{align}
Fix $T>0$.
For every $\varphi\in C_c^\infty([0,T)\times\mR^{2d}\times \mR_+)$, 
by a standard time approximation
and the chain-rule proof of Proposition~\ref{prop:regular-renormalization},
applied to the uniformly elliptic coefficient $a_\eps$, we have
\begin{align}\label{Ap1}
\begin{split}
    &\int_{\mR_+\times\mR^{2d}\times \mR_+}
    \chi_\eps\bigl(\partial_t\varphi+a_\eps(\theta)\Delta_v\varphi
    +v\cdot\nabla_x\varphi\bigr)\dif\theta\dif z\dif t\\
    &\quad-\int_{\mR_+\times\mR^{2d}\times\mR_+}
    \partial_\theta\varphi\,\dif\pi_\eps
    +\int_{\mR^{2d}}\!\int_{\mR_+}
    \chi_0(z,\theta)\varphi(0,z,\theta)\dif\theta\dif z=0.
\end{split}
\end{align}
The total mass of $\pi_\eps$ on every finite time interval is uniformly
bounded by \eqref{uniform-main-estimates}. By weak-* compactness
on each $[0,N]$ and taking a diagonal subsequence, there is a
nonnegative Radon measure $\pi$ on
$[0,\infty)\times\mR^{2d}\times[0,\infty)$ and a
subsequence $\eps_k$ such that $\pi_{\eps_k}\rightharpoonup^*\pi$ vaguely. In
particular, $\pi$ is finite on every time interval. The identity
\begin{align*}
    \int_0^\infty|\chi_\eps-\chi|\dif\theta
    =|f_\eps-f|
\end{align*}
and \eqref{main-strong-L1} imply local $L^1$ convergence of the kinetic
functions.
To justify the coefficient limit, let
$Q\subset[0,T)\times\mR^{2d}$ be compact and let $M<\infty$. Then
by \eqref{approx-local-L1} and the dominated convergence theorem, 
we have as $k\to\infty$,
\begin{align*}
&\int_Q\!\int_0^M
 \bigl|a_{\eps_k}(\theta)\chi_{\eps_k}-a(\theta)\chi\bigr|
 \dif\theta\dif z\dif t\\
&\quad\leq |Q|\int_0^M|a_{\eps_k}(\theta)-a(\theta)|\dif\theta
 +\int_Q
 \bigl|\Phi_a(f_{\eps_k}\wedge M)-\Phi_a(f\wedge M)\bigr|
 \dif z\dif t\longrightarrow0.
\end{align*}
Thus, by taking limits in \eqref{Ap1} along $\eps_k$ for every
admissible test function, we obtain \eqref{kinetic-formulation-distributional}.

It remains to verify the domination property. Let
$0\leq h\in C_c([0,T)\times\mR^{2d}\times\mR_+)$. Here
$h(\cdot,f_\eps)$ denotes
$(t,z)\mapsto h(t,z,f_\eps(t,z))$, and similarly for $f$. 
For every $g\in L^2$, by the dominated convergence, we have
\begin{align*}
    h(\cdot,f_\eps)^{1/2}g
    \longrightarrow h(\cdot,f)^{1/2}g
    \quad\text{strongly in }L^2.
\end{align*}
Consequently,
\begin{align*}
    h(\cdot,f_\eps)^{1/2}
    \nabla_vH_{a_\eps}(f_\eps)
    \rightharpoonup
    h(\cdot,f)^{1/2}\nabla_vH_a(f)
\end{align*}
weakly in $L^2$. By weak lower semicontinuity and
\eqref{approximate-kinetic-measure}, we have
\begin{align*}
    \int h(\cdot,f)|\nabla_vH_a(f)|^2
    &\leq\liminf_{\eps\downarrow0}
      \int h(\cdot,f_\eps)
      |\nabla_vH_{a_\eps}(f_\eps)|^2=\lim_{\eps\downarrow0}\int h\dif\pi_\eps
      =\int h\dif\pi.
\end{align*}
Thus $(f,\pi)$ is a renormalized kinetic solution.
\end{proof}
\subsection{Proof of the probabilistic representation}

\begin{proof}[Proof of Corollary~\ref{cor-probabilistic-representation}]
Fix $T>0$ and write
\begin{align*}
    \mu_t(\dif z):=f(t,z)\dif z.
\end{align*}
Condition \eqref{strict-probabilistic-growth} implies
\eqref{critical-growth-condition}. Hence, by Theorem~\ref{thm:main} and
Proposition~\ref{prop-existence-weak-pde},
each $\mu_t$ is a probability measure. Moreover, if $t_n\to t$, weak
$L^2$ continuity gives convergence of $\mu_{t_n}$ against
$C_c^\infty(\mR^{2d})$. Thus $\mu_{t_n}\to\mu_t$ vaguely; since all
these measures have the same total mass $1$, the convergence is narrow.
Consequently, $(\mu_t)_{t\in[0,T]}$ is a narrowly continuous family of
probability measures.
Let $\Sigma_f$ be defined by \eqref{density-diffusion-coefficient}.
With $z=(x,v)$, define
\begin{align*}
    b(t,z):=(v,0),
    \qquad
    A(t,z):=
    \begin{pmatrix}
        0&0\\
        0&\Sigma_f(t,z)I_d
    \end{pmatrix}.
\end{align*}
The coefficients $b$ and $A$ are Borel measurable, $A$ is symmetric
and nonnegative, and their local integrability with respect to
$\mu_t(\dif z)\dif t$ follows from
$\Phi_a(f)\in L^1_{\mathrm{loc}}$.
Since $\Sigma_f f=\Phi_a(f)$ almost everywhere, the weak formulation
\eqref{Weak1} is the Fokker--Planck equation
\begin{align*}
    \partial_t\mu_t
    =-\operatorname{div}(b\mu_t)
      +\sum_{i,j=1}^{2d}\partial_{z_i}\partial_{z_j}(A_{ij}\mu_t)
\end{align*}
in the sense of distributions.
More explicitly, for every compact $K\subset[0,T]\times\mR^{2d}$,
\begin{align*}
    \int_K\|A(t,z)\|\,\mu_t(\dif z)\dif t
    \leq C\int_K\Phi_a(f(t,z))\dif z\dif t<\infty,
\end{align*}
while the corresponding local integrability of $b$ is immediate.

We verify the weighted integrability condition in the superposition
principle of \cite[Theorem~1.1]{BRS19}; see also \cite{BR20}. 
Note that
\begin{align*}
    \frac{|\langle b(t,z),z\rangle|}{1+|z|^2}
    =\frac{|v\cdot x|}{1+|x|^2+|v|^2}
    \leq\frac12,
\end{align*}
and using \eqref{As17},
\begin{align}\label{weighted-superposition-integrability}
\begin{split}
    \int_0^T\!\int
    \frac{\|A(t,z)\|}{1+|z|^2}\,\mu_t(\dif z)\dif t
    &\leq C\int_0^T\!\int
    \frac{\Phi_a(f(t,z))}{1+|z|^2}\dif z\dif t\leq C\int_0^T\!\int
    \frac{f^{\sigma_0}+f^{\sigma_1}}{1+|z|^2}\dif z\dif t<\infty.
\end{split}
\end{align}
Indeed, $f^{\sigma_1}\leq f+f^2$.  If $\sigma_0<1$, by H\"older's inequality we have
\begin{align*}
    \int_{\mR^{2d}}\frac{f(t,z)^{\sigma_0}}{1+|z|^2}\dif z
    &\leq
    \|f(t)\|_{L^1}^{\sigma_0}
    \left(
        \int_{\mR^{2d}}
        (1+|z|^2)^{-1/(1-\sigma_0)}\dif z
    \right)^{1-\sigma_0}.
\end{align*}
The last integral is finite precisely when
$1/(1-\sigma_0)>d$, which is \eqref{strict-probabilistic-growth}.

The superposition principle now yields a probability measure
$\mathbb P$ on $\mathcal W_T$ whose time marginals are $\mu_t$ and under
which the canonical process solves the martingale problem with drift $b$
and diffusion matrix $A$. Here the generator convention is
$b\cdot\nabla+A:\nabla^2$, so the velocity noise in the corresponding weak
SDE is $\sqrt{2\Sigma_f}\,I_d$. This is precisely
\eqref{martingale-problem}. After enlarging the filtered probability
space if necessary, the standard equivalence with weak solutions of
stochastic differential equations yields
\eqref{formal-density-dependent-sde}.
\end{proof}

\section{Uniqueness: Proof of Theorem~\ref{thm:unique}}\label{sec:uniqueness}

We begin with a localized Kato's type inequality. Its proof uses the
graph-measure domination \eqref{controlKM} and the cutoff-measure
vanishing property \eqref{cutoff-measure-vanishing}; the latter follows
from Proposition~\ref{prop:kinetic-to-weak} for mass-preserving
renormalized kinetic solutions.
\begin{lemma}
\label{lem-localized-kato}
Let $(f_i,\pi_i)$, $i=1,2$, be two mass-preserving renormalized kinetic
solutions such that
\begin{align*}
    \Phi_a(f_i)\in
    L^1_{\mathrm{loc}}(\mR_+\times\mR^{2d}),
    \quad i=1,2.
\end{align*}
Then, for every nonnegative $\psi\in C_c^\infty(\mR_+\times\mR^{2d})$,
\begin{align}
\label{eq-localized-kato-initial}
\begin{split}
    \int_{\mR_+}\!\!\int_{\mR^{2d}}|f_1-f_2|
    &(\partial_t\psi+v\cdot\nabla_x\psi)\dif z\dif t
    +
    \int_{\mR_+}\!\!\int_{\mR^{2d}}
    |\Phi_a(f_1)-\Phi_a(f_2)|\Delta_v\psi\dif z\dif t\\
    &\quad+
    \int_{\mR^{2d}}|f_{1,0}-f_{2,0}|\psi(0,z)\dif z
    \geq0.
\end{split}
\end{align}
\end{lemma}

\begin{proof}
Throughout the proof, for
$(t,z,\theta)\in(0,\infty)\times\mR^{2d}\times\mR$, write for $i=1,2$,
\begin{align*}
    \chi_i(t,z,\theta)
    :=
    \1_{\{0<\theta<f_i(t,z)\}},\quad
    \chi_{i,0}(z,\theta)
    :=
    \1_{\{0<\theta<f_{i,0}(z)\}}.
\end{align*}
We use the elementary identities
\begin{align}\label{key1}
|\chi_1-\chi_2|=|\chi_1-\chi_2|^2=\chi_1+\chi_2-2\chi_1\chi_2,
\end{align}
and for each $(t,z)$,
\begin{align}
\label{eq-kinetic-identities-kato-no-time}
\begin{split}
    \int_0^\infty |\chi_1-\chi_2|(t,z,\theta)\dif\theta
    &=|f_1(t,z)-f_2(t,z)|,\\
    \int_0^\infty a(\theta)|\chi_1-\chi_2|(t,z,\theta)\dif\theta
    &=|\Phi_a(f_1(t,z))-\Phi_a(f_2(t,z))|.
\end{split}
\end{align}
We extend the kinetic formulation
\eqref{kinetic-formulation-distributional} from the half-line to the
entire time axis.
For this argument only, extend $a$ to the whole real line by
setting $a(\theta)=0$ for $\theta\leq0$.
Define
\begin{align*}
    \bar\chi_i(t,z,\theta)
    :=
    \begin{cases}
        \chi_{i,0}(z,\theta),&t\leq0,\\
        \chi_i(t,z,\theta),&t>0,
    \end{cases}
\end{align*}
and
$$
    \begin{cases}
        \bar A_i(t,z,\theta)
    :=\1_{\{\theta>0\}}a(\theta)\bar\chi_i(t,z,\theta),\\
       \bar\pi_i(\dif t,\dif z,\dif \theta):=\1_{\{t\geq 0\}}\pi_i(\dif t,\dif z,\dif \theta).
    \end{cases}
$$
It follows from \eqref{kinetic-formulation-distributional} that
$\bar\chi_i$ satisfies, in the sense of distributions on
$\mR\times\mR^{2d}\times\mR$,
\begin{align}
\label{eq-whole-time-kinetic-kato}
    \partial_t\bar\chi_i+v\cdot\nabla_x\bar\chi_i
    =\Delta_v\bar A_i+\partial_\theta\bar\pi_i
    +\1_{\{t<0\}}G_{i,0},
\end{align}
where, for $z=(x,v)$,
\begin{align*}
    G_{i,0}(z,\theta)
    :=v\cdot\nabla_x\chi_{i,0}(z,\theta)
    -\Delta_v\!\left(
      \1_{\{\theta>0\}}a(\theta)\chi_{i,0}(z,\theta)
    \right).
\end{align*}
Indeed, for every
$\varphi\in C_c^\infty(\mR\times\mR^{2d}\times\mR)$,
\begin{align*}
    \int_{-\infty}^0\chi_{i,0}\partial_t\varphi\dif t
    =\chi_{i,0}\varphi(0),
\end{align*}
which reproduces exactly the initial term in
\eqref{kinetic-formulation-distributional}, while the remaining negative-time
transport and diffusion terms are cancelled by $G_{i,0}$.

We divide the argument into five steps.

\medskip
\noindent
\textbf{Step 1. A regularized identity for $\chi_1+\chi_2-2\chi_1\chi_2$.}
We mollify \eqref{eq-whole-time-kinetic-kato} and derive the localized
identity associated with \eqref{key1}.
Let $\omega\in C_c^\infty((-1,1))$ be a nonnegative even probability
density and set $\omega_\tau(t):=\tau^{-1}\omega(t/\tau)$. Let
$\rho_\varepsilon$ be defined by \eqref{phase-mollifier}. In addition,
choose a nonnegative even
$\vartheta\in C_c^\infty((-1,1))$ with $\int_\mR\vartheta=1$, and set
$\vartheta_\eta(r):=\eta^{-1}\vartheta(r/\eta)$.
Set
\begin{align*}
    K_{\tau,\varepsilon,\eta}(t,z,\theta)
    :=\omega_\tau(t)\rho_\varepsilon(z)\vartheta_\eta(\theta)\geq0.
\end{align*}
For $\tau,\eps,\eta\in(0,1)$, define, with convolution in $(t,z,\theta)$,
$$
    \chi_i^{\tau,\varepsilon,\eta}
    :=
    \bar\chi_i*K_{\tau,\varepsilon,\eta},\quad
    A_i^{\tau,\varepsilon,\eta}
    :=\bar A_i*K_{\tau,\varepsilon,\eta},\quad 
    \Pi_i^{\tau,\varepsilon,\eta}
    :=
    \bar\pi_i*K_{\tau,\varepsilon,\eta},
$$
and
$$
    g_i^{\tau,\varepsilon,\eta}
    :=
    \bigl(\1_{\{t<0\}}G_{i,0}\bigr)*K_{\tau,\varepsilon,\eta}.
$$
Taking convolution for both sides of \eqref{eq-whole-time-kinetic-kato} with $K_{\tau,\varepsilon,\eta}$ gives, in the classical
sense,
\begin{align}
\label{eq-regularized-kinetic-all-time}
    \partial_t\chi_i^{\tau,\varepsilon,\eta}
    +v\cdot\nabla_x\chi_i^{\tau,\varepsilon,\eta}
    =
    \Delta_vA_i^{\tau,\varepsilon,\eta}
    +\partial_\theta\Pi_i^{\tau,\varepsilon,\eta}
    +b_i^{\tau,\varepsilon,\eta}
    +g_i^{\tau,\varepsilon,\eta},
\end{align}
where, for $z=(x,v)$ and $\bar z=(y,w)$,
\begin{align}\label{BB1}
    b_i^{\tau,\varepsilon,\eta}(t,z,\theta)
    :=
    \int_\mR\!\int_{\mR^{2d}}\!\!\int_\mR
    \omega_\tau(t-s)
    (v-w)\cdot\nabla_x\rho_\varepsilon(z-\bar z)
    \vartheta_\eta(\theta-\bar\theta)
    \bar\chi_i(s,\bar z,\bar\theta)
    \dif\bar\theta\dif\bar z\dif s.
\end{align}
Motivated by \eqref{key1}, set
\begin{align*}
    Q^{\tau,\eps,\eta}:=\chi^{\tau,\eps,\eta}_1+\chi^{\tau,\eps,\eta}_2-2\chi^{\tau,\eps,\eta}_1\chi^{\tau,\eps,\eta}_2.
\end{align*}
We temporarily suppress the superscript
$(\tau,\varepsilon,\eta)$. By the chain rule and \eqref{eq-regularized-kinetic-all-time}, we have
\begin{align}
\label{eq-regularized-mismatch-equation-kato}
\begin{split}
    \partial_tQ+v\cdot\nabla_xQ
    &=
    \big[(1-2\chi_2)\Delta_vA_1
       +(1-2\chi_1)\Delta_vA_2\big]+
    (1-2\chi_2)\partial_\theta\Pi_1
    +(1-2\chi_1)\partial_\theta\Pi_2\\
    &\quad+
    \big[(1-2\chi_2)b_1+(1-2\chi_1)b_2\big]+
    \big[(1-2\chi_2)g_1+(1-2\chi_1)g_2\big].
\end{split}
\end{align}
For $h>0$, choose a nondecreasing
$\beta_h\in C^\infty(\mR;[0,1])$ such that
\begin{align*}
    \beta_h=0\quad\text{on }(-\infty,-h],
    \qquad
    \beta_h=1\quad\text{on }[-h/2,\infty),
    \qquad
    \int_{-h}^{-h/2}\beta_h'(t)\dif t=1.
\end{align*}
Fix a nonnegative
$\psi\in C_c^\infty(\mR_+\times\mR^{2d})$ and choose $T>1$ so
that
\begin{align*}
    \operatorname{supp}_t\psi\subset[0,T-1).
\end{align*}
We extend $\psi$ continuously to negative times by
\begin{align*}
    \widehat\psi(t,z):=
    \begin{cases}
        \psi(0,z),&t<0,\\
        \psi(t,z),&t\geq0.
    \end{cases}
\end{align*}
Define
$$
\Psi_h(t,z):=\beta_h(t)\widehat\psi(t,z).
$$
Then 
$\p_t\Psi_h\in L^\infty(\mR\times\mR^{2d})$; it need not be smooth across
$t=0$. All integrations by parts below remain valid by approximating
$\Psi_h$  in the time variable only with nonnegative smooth
compactly supported functions.

Let $\zeta_\delta$ be defined by \eqref{kinetic-cutoff}.
Multiplying \eqref{eq-regularized-mismatch-equation-kato} by $\Psi_h(t,z)\zeta_\delta(\theta)$ and integrating by parts
in time, $x$, $v$, and $\theta$, we obtain
\begin{align}
\label{eq-mismatch-master-all-time}
\begin{split}
    \cL_{\tau,\varepsilon,\eta,\delta,h}&:=\int Q\zeta_\delta
    (\partial_t\Psi_h+v\cdot\nabla_x\Psi_h)
    +\int a(\theta)Q\zeta_\delta\Delta_v\Psi_h\\
    &=
    \mathcal E_{\tau,\varepsilon,\eta,\delta,h}
    -2\mathcal D_{\tau,\varepsilon,\eta,\delta,h}
    +\mathcal C_{\tau,\varepsilon,\eta,\delta,h}
    -\mathcal B_{\tau,\varepsilon,\eta,\delta,h}
    -\mathcal G_{\tau,\varepsilon,\eta,\delta,h},
\end{split}
\end{align}
where unmarked integrals are over
the whole $\mR\times\mR^{2d}\times\mR$ with respect to
$\dif t\dif z\dif\theta$, and
\begin{align*}
    \mathcal E_{\tau,\varepsilon,\eta,\delta,h}
    &:=
    \int\zeta_\delta
    \big[
    (1-2\chi_2)(\nabla_vA_1-a(\theta)\nabla_v\chi_1)
    +(1-2\chi_1)(\nabla_vA_2-a(\theta)\nabla_v\chi_2)
    \big]\cdot\nabla_v\Psi_h,\\
    \mathcal D_{\tau,\varepsilon,\eta,\delta,h}
    &:=
    \int\Psi_h\zeta_\delta
    \big(
    \nabla_v\chi_2\cdot\nabla_vA_1
    +\nabla_v\chi_1\cdot\nabla_vA_2
    +\partial_\theta\chi_2\,\Pi_1
    +\partial_\theta\chi_1\,\Pi_2
    \big),\\
    \mathcal C_{\tau,\varepsilon,\eta,\delta,h}
    &:=
    \int\Psi_h\zeta_\delta'
    \big[(1-2\chi_2)\Pi_1+(1-2\chi_1)\Pi_2\big],\\
    \mathcal B_{\tau,\varepsilon,\eta,\delta,h}
    &:=
    \int\Psi_h\zeta_\delta
    \big[(1-2\chi_2)b_1+(1-2\chi_1)b_2\big],\\
    \mathcal G_{\tau,\varepsilon,\eta,\delta,h}
    &:=
    \int\Psi_h\zeta_\delta
    \big[(1-2\chi_2)g_1+(1-2\chi_1)g_2\big].
\end{align*}
Indeed, by integration by parts in $v$, we have
\begin{align*}
    &\int\Psi_h\zeta_\delta
    \big[(1-2\chi_2)\Delta_vA_1
        +(1-2\chi_1)\Delta_vA_2\big]\\
    &\quad=
    \int a(\theta)Q\zeta_\delta\Delta_v\Psi_h
    +2\int\Psi_h\zeta_\delta
    \big(
    \nabla_v\chi_2\cdot\nabla_vA_1
    +\nabla_v\chi_1\cdot\nabla_vA_2
    \big)
    -\mathcal E_{\tau,\varepsilon,\eta,\delta,h},
\end{align*}
where we have used
\begin{align*}
    (1-2\chi_2)\nabla_v\chi_1
    +(1-2\chi_1)\nabla_v\chi_2
    =\nabla_vQ.
\end{align*}
Similarly, 
\begin{align*}
    &\int\Psi_h\zeta_\delta(1-2\chi_2)\partial_\theta\Pi_1
    +\int\Psi_h\zeta_\delta(1-2\chi_1)\partial_\theta\Pi_2\\
    &\quad=
    -\int\Psi_h\zeta_\delta'(1-2\chi_2){\Pi_1}
    -\int\Psi_h\zeta_\delta'(1-2\chi_1){\Pi_2}\\
    &\qquad+
    2\int\Psi_h\zeta_\delta\partial_\theta\chi_2{\Pi_1}
    +2\int\Psi_h\zeta_\delta\partial_\theta\chi_1{\Pi_2}.
\end{align*}
Combining these two identities yields \eqref{eq-mismatch-master-all-time}.
We estimate the five error terms in the next three steps.

\medskip
\noindent
\textbf{Step 2. Estimates of $\cB,\cC,\cG$.}
We begin with the transport commutator $\cB$. Using \eqref{BB1},
$|\bar\chi_i|\leq1$, and \eqref{phase-mollifier}, we have
$$
|b_i^{\tau,\varepsilon,\eta}(t,z,\theta)|
\leq\int_{\mR^{2d}}
    |(v-w)\cdot \nabla_x\rho_\varepsilon(z-\bar z)|\dif\bar z
\leq C_\rho\eps.
$$
Thus, for fixed $\eta$ and $\delta$, since
$|1-2\chi^{\tau,\varepsilon,\eta}_i|\leq1$, we have
\begin{align}
\label{eq-transport-error-all-time}
    \lim_{\varepsilon\downarrow0}
    \sup_{\tau,h\in(0,1)}
    |\mathcal B_{\tau,\varepsilon,\eta,\delta,h}|=0.
\end{align}
For $\cC$, since convolution in time and phase space preserves total mass,
and our choice of $T$ accounts for the unit enlargement of the time support, 
by $|1-2\chi^{\tau,\varepsilon,\eta}_i|\leq1$ again, we have
$$    
\sup_{\tau,\eps,h}|\mathcal C_{\tau,\varepsilon,\eta,\delta,h}|
    \leq
    \|\psi\|_\infty
    \sum_{i=1}^2
    \int_{[0,T]\times\mR^{2d}\times\mR_+}
    \bigl(|\zeta_\delta'|*\vartheta_\eta\bigr)(\theta)
    \dif\pi_i.
$$ 
In particular, for fixed $\delta>0$, by the dominated convergence theorem,
\begin{align}
\label{eq-cutoff-error-kato}
\lim_{\eta\downarrow0}    \sup_{\tau,\eps,h}|\mathcal C_{\tau,\varepsilon,\eta,\delta,h}|
\leq\|\psi\|_\infty
    \sum_{i=1}^2
    \int_{[0,T]\times\mR^{2d}\times\mR_+}
|\zeta_\delta'|(\theta)
    \dif\pi_i.
\end{align}
We finally estimate the negative-time residual $\mathcal G$. Note that
$$
g_i^{\tau,\varepsilon,\eta}(t,z,\theta)
    =q_\tau(t)\bigl(G_{i,0}*
    (\rho_\varepsilon\vartheta_\eta)\bigr)(z,\theta),
$$
where $q_\tau(t):=(\1_{(-\infty,0)}*\omega_\tau)(t)$ satisfies
$$
q_\tau(t)=1\quad(t\leq-\tau),
\qquad q_\tau(t)=0\quad(t\geq\tau).
$$
By the definition of $\mathcal G_{\tau,\varepsilon,\eta,\delta,h}$, we have
\begin{align*}
    |\mathcal G_{\tau,\varepsilon,\eta,\delta,h}|
    &\leq
    3\|\psi\|_\infty
    \sum_{i=1}^2
    \int_{-h}^{\tau}\!
    \int_{\operatorname{supp}_{z,\theta}(\Psi_h\zeta_\delta)}
    |G_{i,0}*
    \bigl(\rho_\varepsilon\vartheta_\eta\bigr)|(z,\theta)
    \dif z\dif\theta\dif t\leq C_{\varepsilon,\eta,\delta,\psi}(h+\tau).
\end{align*}
Hence, for each $\varepsilon,\eta,\delta$,
\begin{align}\label{eq-initial-residual-error-kato}
{\lim_{h\downarrow0}\lim_{\tau\downarrow0}}
|\mathcal G_{\tau,\varepsilon,\eta,\delta,h}|=0.
\end{align}

\medskip
\noindent
\textbf{Step 3. Estimate of $\cD$.}
Fix $0<\eta<\delta/16$. Split
$\cD_{\tau,\varepsilon,\eta,\delta,h}$ into its part on
$\{t>\tau\}$, denoted by
$\cD^+_{\tau,\varepsilon,\eta,\delta,h}$, and the initial-layer part $\cD^-_{\tau,\varepsilon,\eta,\delta,h}$.
On
$(\tau,\infty)\times\mR^{2d}\times\operatorname{supp}\zeta_\delta$,
we have the claim
\begin{align}
\label{eq-cross-measure-control-no-time}
\begin{split}
    &-\partial_\theta\chi_2^{\tau,\varepsilon,\eta}
    \Pi_1^{\tau,\varepsilon,\eta}
    -\partial_\theta\chi_1^{\tau,\varepsilon,\eta}
    \Pi_2^{\tau,\varepsilon,\eta}\geq
    2P_1^{\tau,\varepsilon,\eta}\cdot
    P_2^{\tau,\varepsilon,\eta},
\end{split}
\end{align}
where
$$
P_i^{\tau,\varepsilon,\eta}
    :=
    (\nabla_vH_a(f_i)\mu_i)*
    K_{\tau,\varepsilon,\eta}.
    $$
To prove this claim, we introduce the graph measure
\begin{align*}
   \mu_i(\dif t,\dif z,\dif\theta)
    :=
    \1_{\{t>0\}}\delta_{f_i(t,z)}(\dif\theta)\dif z\dif t,
\end{align*}
and set
\begin{align*}
    W_i^{\tau,\varepsilon,\eta}
     :=
    (|\nabla_vH_a(f_i)|^2\mu_i)*
    K_{\tau,\varepsilon,\eta}.
\end{align*}
On $(\tau,\infty)\times\mR^{2d}\times\operatorname{supp}\zeta_\delta$, we have
$$
m_i^{\tau,\varepsilon,\eta}(t,z,\theta):=-\partial_\theta\chi_i^{\tau,\varepsilon,\eta}(t,z,\theta)
    =[\mu_i*K_{\tau,\varepsilon,\eta}](t,z,\theta)\geq 0,
$$
and by the Cauchy--Schwarz inequality for convolution measures,
\begin{align}
\label{eq-pointwise-measure-control-no-time}
    |P_i^{\tau,\varepsilon,\eta}|^2
    \leq
    (\mu_i*K_{\tau,\varepsilon,\eta})
    W_i^{\tau,\varepsilon,\eta}=m_i^{\tau,\varepsilon,\eta}W_i^{\tau,\varepsilon,\eta}.
\end{align}
Moreover, convolution preserves \eqref{controlKM}, so on the same set
$$
    W_i^{\tau,\varepsilon,\eta}(t,z,\theta)
    \leq\Pi_i^{\tau,\varepsilon,\eta}(t,z,\theta).
$$
Therefore, the left-hand side of
\eqref{eq-cross-measure-control-no-time} is bounded from below by
\begin{align*}
    m_2^{\tau,\varepsilon,\eta}
    W_1^{\tau,\varepsilon,\eta}
    +m_1^{\tau,\varepsilon,\eta}
    W_2^{\tau,\varepsilon,\eta}
    &\geq
    2\sqrt{
      m_1^{\tau,\varepsilon,\eta}
      W_1^{\tau,\varepsilon,\eta}
      m_2^{\tau,\varepsilon,\eta}
      W_2^{\tau,\varepsilon,\eta}}
    \geq
    2P_1^{\tau,\varepsilon,\eta}\cdot
      P_2^{\tau,\varepsilon,\eta}.
\end{align*}
Since $\Psi_h=\psi$ on $\mR_+$ and
$\psi\zeta_\delta\geq0$,
equation~\eqref{eq-cross-measure-control-no-time} implies
\begin{align}
\label{eq-D-upper-bound-kato}
\begin{split}
    \mathcal D_{\tau,\varepsilon,\eta,\delta,h}^{+}
    \leq
    \int_{\{t>\tau\}}\psi\zeta_\delta|\sI_{\tau,\varepsilon,\eta}|,
\end{split}
\end{align}
where
$$
\sI_{\tau,\varepsilon,\eta}:=\nabla_v\chi^{\tau,\varepsilon,\eta}_2\cdot\nabla_vA^{\tau,\varepsilon,\eta}_1
    +\nabla_v\chi^{\tau,\varepsilon,\eta}_1\cdot\nabla_vA^{\tau,\varepsilon,\eta}_2
    -2P^{\tau,\varepsilon,\eta}_1\cdot P^{\tau,\varepsilon,\eta}_2.
$$
Since $a$ and $a^{-1}$ are bounded on
$I_\delta:=[\delta/8,\delta^{-1}+2]$, on the set $\{f_i\in I_\delta\}$, we have
\begin{align*}
\nabla_v\1_{\{\theta<f_i\}}=\nabla_v\1_{\{\theta<H_a^{-1}(H_a(f_i))\}}
&=\delta_{\theta=f_i}(H_a^{-1})'(H_a(f_i))\nabla_vH_a(f_i)=\delta_{\theta=f_i}a(f_i)^{-1/2}\nabla_vH_a(f_i).
\end{align*}
Hence, for
$(t,z,\theta)\in(\tau,\infty)\times\mR^{2d}
\times\operatorname{supp}\zeta_\delta$,
\begin{align}\label{eq-kato-gradient-chi}
\nabla_v\chi_i^{\tau,\varepsilon,\eta}(t,z,\theta)=
    [(\1_{I_\delta}(f_i)
      a(f_i)^{-1/2}\nabla_vH_a(f_i)\mu_i)
    *K_{\tau,\varepsilon,\eta}](t,z,\theta).
\end{align}
Similarly,
\begin{align}\label{eq-kato-gradient-A}
\nabla_vA_i^{\tau,\varepsilon,\eta}(t,z,\theta)
    =
    [({\1_{I_\delta}(f_i)}
      a(f_i)^{1/2}\nabla_vH_a(f_i)\mu_i)
    *K_{\tau,\varepsilon,\eta}](t,z,\theta).
\end{align}
Note that by expanding the convolutions,
\begin{align}\label{eq-kato-coefficient-overlap}
\begin{split}
\sI_{\tau,\varepsilon,\eta}(t,z,\theta)
&=\int_{(0,\infty)^2}\!
  \int_{(\mR^{2d})^2}\!\int_{I_\delta^2}
  \Lambda(\theta_1,\theta_2)
  \nabla_vH_a(f_1)(s,z_1)\cdot
  \nabla_vH_a(f_2)(r,z_2)\\
&\quad\quad\times\omega_\tau(t-s)\omega_\tau(t-r)
  \delta_{f_1(s,z_1)}(\dif\theta_1)
  \delta_{f_2(r,z_2)}(\dif\theta_2)\\
&\quad\quad\times
  \rho_\eps(z-z_1)\rho_\eps(z-z_2)
  \vartheta_\eta(\theta-\theta_1)
  \vartheta_\eta(\theta-\theta_2)
  \dif z_1\dif z_2\dif s\dif r,
\end{split}
\end{align}
where
\begin{align*}
    \Lambda(\theta_1,\theta_2)
    &:=
    \sqrt{\frac{a(\theta_1)}{a(\theta_2)}}
    +\sqrt{\frac{a(\theta_2)}{a(\theta_1)}}-2=
    \frac{(\sqrt{a(\theta_1)}-\sqrt{a(\theta_2)})^2}
    {\sqrt{a(\theta_1)a(\theta_2)}}.
\end{align*}
Since $\sup_{\theta\in I_\delta}|\sqrt{a(\theta)})'|\leq C_\delta$, we have
\begin{align*}
    0\leq\Lambda(\theta_1,\theta_2)
    \lesssim_\delta|\theta_1-\theta_2|^2,\ \ \theta_1,\theta_2\in I_\delta.
\end{align*}
Moreover, if $\vartheta_\eta(\theta-\theta_1)
  \vartheta_\eta(\theta-\theta_2)\not=0$, 
  then $|\theta_1-\theta_2|\leq2\eta$ and
\begin{align*}
    \int_\mR
    \vartheta_\eta(\theta-\theta_1)
    \vartheta_\eta(\theta-\theta_2)\dif\theta
    \lesssim
    \eta^{-1}\1_{\{|\theta_1-\theta_2|\leq2\eta\}}.
\end{align*}
Thus, by \eqref{eq-D-upper-bound-kato} and
\eqref{eq-kato-coefficient-overlap}, Fubini's theorem,  and the Cauchy--Schwarz inequality, we have
\begin{align}
\label{eq-cross-coefficient-error-no-time}
    \mathcal D_{\tau,\varepsilon,\eta,\delta,h}^{+}\lesssim_{\delta,\psi}
    \eta
    {\|\nabla_vH_a(f_1)\|_{L^2([0,T]\times\mR^{2d})}
    \|\nabla_vH_a(f_2)\|_{L^2([0,T]\times\mR^{2d})}.}
\end{align}
Hence,
\begin{align}
\label{eq-D-limit-positive-time}
    \varlimsup_{\eta\downarrow0}
    \sup_{\tau,\varepsilon\in(0,1)}
    \mathcal D_{\tau,\varepsilon,\eta,\delta,h}^{+}\leq0.
\end{align}
For the initial layer $\mathcal D_{\tau,\varepsilon,\eta,\delta,h}^{-}=\mathcal D_{\tau,\varepsilon,\eta,\delta,h}
    -\mathcal D_{\tau,\varepsilon,\eta,\delta,h}^{+}$, the measure cross terms are nonpositive because
$-\partial_\theta\chi_i^{\tau,\varepsilon,\eta}\geq0$ on
$\operatorname{supp}\zeta_\delta$ and
$\Pi_i^{\tau,\varepsilon,\eta}\geq0$. Once
$\varepsilon,\eta,\delta$ are fixed, the remaining gradient terms are
uniformly bounded on the relevant compact set and are integrated over
a time interval of length at most $h+2\tau$. Thus
\begin{align}
\label{eq-initial-layer-D-kato}
    \mathcal D_{\tau,\varepsilon,\eta,\delta,h}^{-}
    \leq C_{\varepsilon,\eta,\delta,\psi}(h+\tau).
\end{align}
Combining \eqref{eq-D-limit-positive-time} and
\eqref{eq-initial-layer-D-kato}, we get
\begin{align}\label{Ak9}
\varlimsup_{\eta\downarrow0}\sup_{\varepsilon\in(0,1)}
{\lim_{h\downarrow0}\lim_{\tau\downarrow0}}
    \mathcal D_{\tau,\varepsilon,\eta,\delta,h}\leq0.
\end{align}

\medskip
\noindent
\textbf{Step 4. Estimate of $\cE$.}
{Choose a compact set $K\Subset\mR^{2d}$ such that
$\operatorname{supp}_z\Psi_h\subset K$.} By definition,
\begin{align*}
|\mathcal E_{\tau,\varepsilon,\eta,\delta,h}|
    &\leq C_\psi\sum_{i=1,2}
    \left(\int_{[-h,\tau]\times K\times \operatorname{supp}\zeta_\delta}+\int_{[\tau,T]\times K\times \operatorname{supp}\zeta_\delta}\right)
    |\nabla_vA^{\tau,\varepsilon,\eta}_i-a(\theta)\nabla_v\chi^{\tau,\varepsilon,\eta}_i|.
\end{align*}
For $t>\tau$ and $\theta\in\operatorname{supp}\zeta_\delta$,
by \eqref{eq-kato-gradient-chi} and \eqref{eq-kato-gradient-A}, we have
\begin{align*}
    &\nabla_vA_i^{\tau,\varepsilon,\eta}(t,z,\theta)
    -a(\theta)\nabla_v\chi_i^{\tau,\varepsilon,\eta}(t,z,\theta)\\
    &=
    \int_\mR\!\int_{\mR^{2d}}\1_{I_\delta}(f_i(s,\bar z))
    \frac{a(f_i(s,\bar z))-a(\theta)}
    {\sqrt{a(f_i(s,\bar z))}}
    \nabla_vH_a(f_i)(s,\bar z)
    \omega_\tau(t-s)
    \rho_\varepsilon(z-\bar z)
    \vartheta_\eta(\theta-f_i(s,\bar z))
    \dif\bar z\dif s.
\end{align*}
Since {$a$, $a^{-1}$, and $a'$ are bounded} on
$I_\delta=[\delta/8,\delta^{-1}+2]$,
we have on the support of the integrand,
\begin{align*}
    \left|
    \frac{a(f_i(s,\bar z))-a(\theta)}
    {\sqrt{a(f_i(s,\bar z))}}
    \right|
    \lesssim_\delta\eta.
\end{align*}
By Fubini's theorem and the Cauchy--Schwarz
inequality, we have
\begin{align*}
 \sup_{\tau,\varepsilon\in(0,1)}\int_{[\tau,T]\times K\times \operatorname{supp}\zeta_\delta}
    |\nabla_vA^{\tau,\varepsilon,\eta}_i-a(\theta)\nabla_v\chi^{\tau,\varepsilon,\eta}_i|
    \leq
    C_{\delta,\psi}\eta
    {\|\nabla_vH_a(f_i)\|_{L^2([0,T+1]\times\mR^{2d})}.}
\end{align*}
Moreover,
$$
\int_{[-h,\tau]\times K\times \operatorname{supp}\zeta_\delta}    |\nabla_vA^{\tau,\varepsilon,\eta}_i-a(\theta)\nabla_v\chi^{\tau,\varepsilon,\eta}_i|
\leq C_{\varepsilon,\eta,\delta,\psi}(h+\tau).
$$
Combining the above estimates, we obtain
\begin{align}
\label{eq-coefficient-error-positive-time}
    \lim_{\eta\downarrow0}\sup_{\varepsilon\in(0,1)}
    {\lim_{h\downarrow0}\lim_{\tau\downarrow0}}
    |\mathcal E_{\tau,\varepsilon,\eta,\delta,h}|=0.
\end{align}

\medskip
\noindent
\textbf{Step 5. Passage to the limit.}
By \eqref{eq-mismatch-master-all-time}, \eqref{eq-transport-error-all-time}, \eqref{eq-cutoff-error-kato}, 
\eqref{eq-initial-residual-error-kato}, \eqref{Ak9}, and
{\eqref{eq-coefficient-error-positive-time}}, we have
\begin{align}\label{Ak7}
\varliminf_{\eta\downarrow 0}\varliminf_{\eps\downarrow 0}\lim_{h\downarrow 0}\lim_{\tau\downarrow 0}
\cL_{\tau,\varepsilon,\eta,\delta,h}\geq-\|\psi\|_\infty
    \sum_{i=1}^2
    {\int_{[0,T]\times\mR^{2d}\times\mR_+}}
    |\zeta_\delta'(\theta)|\dif\pi_i.
\end{align}
Note that by the definition of $\cL$ and $\Psi_h=\beta_h\widehat\psi$,
\begin{align*}
\cL_{\tau,\varepsilon,\eta,\delta,h}
&=\int Q^{\tau,\eps,\eta}\zeta_\delta
    (\partial_t\Psi_h+v\cdot\nabla_x\Psi_h+a(\theta)\Delta_v\Psi_h)\\
&=\int Q^{\tau,\eps,\eta}\zeta_\delta
    \Big(\beta_h'\widehat\psi+\beta_h(\partial_t\widehat\psi+v\cdot\nabla_x\widehat\psi+a(\theta)\Delta_v\widehat\psi)\Big).
\end{align*}
For $\tau<h/8$ and $0<\eta<\delta/16$, one has
\begin{align}
\label{eq-initial-term-from-negative-time}
\begin{split}
    &\int_{-h}^{-h/2}\!\int_{\mR^{2d}}\!\int_\mR
    Q^{\tau,\varepsilon,\eta}
    \zeta_\delta\beta_h'(t)\psi(0,z)
    \dif\theta\dif z\dif t=
    \int_{\mR^{2d}}\!\int_\mR
    Q_0^{\varepsilon,\eta}(z,\theta)
    \zeta_\delta(\theta)\psi(0,z)\dif\theta\dif z,
\end{split}
\end{align}
where
\begin{align*}
    \chi_{i,0}^{\varepsilon,\eta}
    &:=
    \chi_{i,0}*_{z,\theta}
      (\rho_\varepsilon\vartheta_\eta),\quad\
    Q_0^{\varepsilon,\eta}
    :=
    \chi_{1,0}^{\varepsilon,\eta}
    +\chi_{2,0}^{\varepsilon,\eta}
    -2\chi_{1,0}^{\varepsilon,\eta}
       \chi_{2,0}^{\varepsilon,\eta}.
\end{align*}
By the approximation-of-the-identity property and
$0\leq\chi_i\leq1$, we have
\begin{align*}
    Q^{\tau,\varepsilon,\eta}
    &\longrightarrow|\chi_1-\chi_2|,\quad\
    Q_0^{\varepsilon,\eta}
    \longrightarrow|\chi_{1,0}-\chi_{2,0}|\mbox{ in $L^1_{\mathrm{loc}}$}
\end{align*}
as
$\tau,\varepsilon,\eta\downarrow0$ in the prescribed order.
The contribution over $(-h,0)$ that does not contain $\beta_h'$ is
$O_{\varepsilon,\eta,\delta,\psi}(h)$ and therefore vanishes before
the spatial and kinetic mollification parameters are removed.
Consequently,
\begin{align}
\label{eq-kato-with-delta-cutoff}
\begin{split}
    \varliminf_{\eta\downarrow 0}\varliminf_{\eps\downarrow 0}\lim_{h\downarrow 0}\lim_{\tau\downarrow 0}
\cL_{\tau,\varepsilon,\eta,\delta,h}
&=\int_{\mR_+}\!\!\int_{\mR^{2d}}\!\!\int_0^\infty
    |\chi_1-\chi_2|\zeta_\delta
    (\partial_t\psi+v\cdot\nabla_x\psi)
    \dif\theta\dif z\dif t\\
    &\quad+
    \int_{\mR_+}\!\!\int_{\mR^{2d}}\!\!\int_0^\infty
    a(\theta)|\chi_1-\chi_2|\zeta_\delta\Delta_v\psi
    \dif\theta\dif z\dif t\\
    &\quad+
    \int_{\mR^{2d}}\!\!\int_0^\infty
    |\chi_{1,0}-\chi_{2,0}|
    \zeta_\delta(\theta)\psi(0,z)
    \dif\theta\dif z,
\end{split}
\end{align}
Combining this identity with \eqref{Ak7} and then letting
$\delta\downarrow0$, and by \eqref{cutoff-measure-vanishing}, we get
\begin{align*}
\int_{\mR_+}\!\!\int_{\mR^{2d}}\!\!\int_0^\infty
    |\chi_1-\chi_2|
    (\partial_t\psi+v\cdot\nabla_x\psi+a(\theta)\Delta_v\psi)
    \dif\theta\dif z\dif t+
    \int_{\mR^{2d}}\!\!\int_0^\infty
    |\chi_{1,0}-\chi_{2,0}|\psi(0,z)
    \dif\theta\dif z\geq 0,
\end{align*}
which in turn yields \eqref{eq-localized-kato-initial} by \eqref{eq-kinetic-identities-kato-no-time}.
\end{proof}

\begin{proof}[Proof of Theorem~\ref{thm:unique}]
By \eqref{As17}, $\Phi_a$ satisfies \eqref{mass-growth-condition} with
$\sigma=\sigma_0$. Thus, by \eqref{critical-growth-condition} and
Proposition~\ref{prop-existence-weak-pde}, the solution
constructed in Theorem~\ref{thm:main} preserves mass.

It remains to prove \eqref{Contra}. {Fix $T>0$}, and let $f_1$ and $f_2$
be two mass-preserving renormalized kinetic solutions in the class
specified in the theorem. By Assumption~\ref{A3}, we have
$\Phi_a(f_i)\in L^1_{\mathrm{loc}}$. Set
\begin{align*}
    u:=f_1-f_2,
    \qquad
    u_0:=f_{1,0}-f_{2,0},
    \qquad
    B:=\Phi_a(f_1)-\Phi_a(f_2).
\end{align*}
Let $\gamma_R$ be defined by \eqref{Ga1}, and set
\begin{align*}
    \phi_R(x,v):=\gamma_R(x)\gamma_R(v),
    \qquad R\geq1.
\end{align*}
Then $0\leq\phi_R\leq1$, $\phi_R\to1$ pointwise, and
\begin{align}
\label{eq-contraction-cutoff-bounds}
\begin{split}
    |\Delta_v\phi_R|
    &\leq
C R^{-2}
    \1_{\{|x|\leq2R,\ R\leq|v|\leq2R\}},\quad\
    |v\cdot\nabla_x\phi_R|
 \leq
    C\1_{\{R\leq|x|\leq2R,\ |v|\leq2R\}}.
\end{split}
\end{align}
For each integer $R\geq1$, define
\begin{align*}
    F_R(s):=
    \int_{\mR^{2d}}|u(s,z)|\phi_R(z)\dif z.
\end{align*}
Let $I\subset(0,T)$ be the common
full-measure set of Lebesgue points of $F_R$, $R\in\mN$.
Fix $t\in I$ and $R\in\mN$. Choose a nonnegative
$\kappa\in C_c^\infty((-1,0))$ with
{$\int_\mR\kappa=1$}, and, for sufficiently
large $n$, set
\begin{align*}
    \omega_{n,t}(s):=n\kappa(n(s-t)),
    \qquad
    \eta_{n,t}(s):=\int_s^\infty \omega_{n,t}(r)\dif r.
\end{align*}
Thus $\eta_{n,t}\in C_c^\infty(\mR_+;[0,1])$,
$\eta_{n,t}(0)=1$, $\eta_{n,t}'=-\omega_{n,t}$, and
$\eta_{n,t}\to\1_{[0,t)}$ almost everywhere. By 
Lemma~\ref{lem-localized-kato} with
$\psi(s,z)=\eta_{n,t}(s)\phi_R(z)$ and then letting $n\to\infty$, we have for each $t\in I$,
\begin{align}
\label{eq-kato-ae-time-cutoff}
\begin{split}
    \int_{\mR^{2d}}|u(t)|\phi_R\dif z
    &\leq
    \int_{\mR^{2d}}|u_0|\phi_R\dif z
    +\int_0^t\!\!\int_{\mR^{2d}}
    |B|\Delta_v\phi_R\dif z\dif s+
    \int_0^t\!\!\int_{\mR^{2d}}
    |u|v\cdot\nabla_x\phi_R\dif z\dif s.
\end{split}
\end{align}
As in the proof of Proposition \ref{prop-existence-weak-pde},
letting $R\to\infty$ and using dominated convergence, we obtain
\begin{align}
\label{eq-contraction-ae-time}
    \|u(t)\|_{L^1}
    \leq
    \|u_0\|_{L^1},
    \qquad t\in {I}.
\end{align}
For any $t\in[0,T]$, choose $t_n\in{I}$ such that $t_n\to t$.
Since
renormalized kinetic solutions belong to $C_w([0,T];L^2)$, we have
\begin{align*}
    u(t_n)\rightharpoonup u(t)
    \qquad\text{weakly in }L^2(\mR^{2d}).
\end{align*}
For every $N>0$, let $B_N$ denote the ball of radius $N$ in
$\mR^{2d}$. By the weak lower semicontinuity of $L^1$-norm, we have
\begin{align*}
    \int_{B_N}|u(t)|
    &\leq
    \liminf_{n\to\infty}\int_{B_N}|u(t_n)|
    \leq
    \|u_0\|_{L^1}.
\end{align*}
Letting $N\to\infty$ , we get \eqref{Contra} for every $t\in[0,T]$.
\end{proof}
\appendix

\section{Approximation of the diffusion coefficient}
\label{appendix:diffusion-approximation}

\begin{lemma}[Approximation of the diffusion coefficient]
\label{approx-Psi}
Let
$
    a\in C^1((0,\infty);(0,\infty))
      \cap L^1_{\mathrm{loc}}([0,\infty))
$
satisfy \eqref{As15}--\eqref{As16}. Then there exists a family
$(a_\eps)_{\eps\in(0,1)}\subset C_b^1(\mR;(0,\infty))$ with the
following properties.
For each $\eps$, there exist
$0<\underline a_\eps\leq\overline a_\eps<\infty$ such that
\begin{align}\label{approx-uniform-ellipticity}
    \underline a_\eps\leq a_\eps(\theta)
    \leq\overline a_\eps,
    \qquad \theta\in\mR,
\end{align}
and there is a constant $C$, independent of $\eps$, such that
\begin{align}\label{approx-log-estimate}
    |(\ln a_\eps(\theta))'|
    &\leq C\kappa_0
    (\theta^{q_0-2}+\theta^{q_1-2})
    (a_\eps(\theta)^\ell\wedge a_\eps(\theta)^{\ell/3}),
    \qquad \theta>0,
\end{align}
and
\begin{align}\label{approx-integral-estimate}
    \int_0^\theta
    (a_\eps(s)^{-\ell}\vee a_\eps(s)^{-\ell/3})\dif s
    \leq C(\theta+\theta^2),
    \qquad \theta\geq0.
\end{align}
Moreover, as $\eps\to0$,
\begin{align}\label{approx-local-uniform}
    a_\eps\longrightarrow a
    \quad\text{locally uniformly on }(0,\infty),
\end{align}
and, for every $R>0$,
\begin{align}\label{approx-local-L1}
    \int_0^R|a_\eps(s)-a(s)|\dif s\longrightarrow0,
\end{align}
and there is a $\lambda_\eps>0$ with $\lambda_\eps\downarrow0$, such that for every $s\geq0$,
\begin{align}\label{primitive-comparison}
    \Phi_{a_\eps}(s)\leq\Phi_a(s)+\lambda_\eps s+2\eps.
\end{align}
\end{lemma}

\begin{proof}
Define $\mathfrak B(s):=s^{-\ell}+s^{-\ell/3}$ for $s>0$. Let
$\mathfrak A=\mathfrak B^{-1}$ be the inverse of $\cB$. Define
\begin{align*}
    b(\theta):=\mathfrak B(a(\theta)).
\end{align*}
By \eqref{As15}, we have
\begin{align}\label{b-derivative-bound}
    |b'(\theta)|
    =\left(\ell a(\theta)^{-\ell}
      +\tfrac\ell3a(\theta)^{-\ell/3}\right)
      |(\ln a(\theta))'|
      \leq \tfrac{4\ell}{3}\kappa_0(\theta^{q_0-2}+\theta^{q_1-2}).
\end{align}
The right-hand side is integrable near zero. Hence $b$ has the finite
limit
\begin{align*}
    b_0:=\lim_{\theta\downarrow0}b(\theta).
\end{align*}
Moreover,
by \eqref{As16}, we have
\begin{align*}
    \frac1\theta\int_0^\theta b(r)\dif r
    &\leq
    \frac2\theta\int_0^\theta
      \bigl(a(r)^{-\ell}\vee a(r)^{-\ell/3}\bigr)\dif r
    \leq2\kappa_1(1+\theta).
\end{align*}
Letting $\theta\downarrow0$ gives $0\leq b_0\leq2\kappa_1$.
Integrating
\eqref{b-derivative-bound} from zero, we get
\begin{align}\label{b-linear-growth}
    b(\theta)\leq C(1+\kappa_0+\kappa_1)(1+\theta).
\end{align}
We extend $b$ continuously to zero by $b(0)=b_0$.

Now choose a nondecreasing $h\in C^\infty([0,\infty))$ such that
\begin{align}\label{h-eps-properties}
    h(s)=1\quad\left(s\leq\tfrac12\right),
    \quad
    h(s)=s\quad(s\geq2),\quad
    s\leq h(s)\leq s+1,
    \quad 0\leq h'(s)\leq C.
\end{align}
Choose also a nondecreasing $\varpi\in C^\infty([0,\infty))$ satisfying
\begin{align}\label{Ab8}
    \varpi(s)=s\quad(s\leq1),
    \qquad
    \varpi(s)=2\quad(s\geq3),
    \qquad
    0\leq\varpi'(s)\leq1,
    \qquad
    0\leq\varpi(s)\leq s.
\end{align}
Let
\begin{align*}
    \Lambda_\eps:=\mathfrak A(\eps),
    \qquad
    \lambda_\eps:=\mathfrak A(1/\eps),
    \qquad
    \delta_\eps:=\eps/(1+\Lambda_\eps).
\end{align*}
Then
\begin{align}\label{delta-Lambda-small}
    \delta_\eps\Lambda_\eps\leq\eps,
    \qquad
    \lambda_\eps\downarrow0.
\end{align}
Define
\begin{align}
    b_\eps(\theta)
    :=\eps+\varpi\Big(\eps b\big(\delta_\eps h\big(\theta/\delta_\eps\big)\big)\Big)/\eps,
    \quad
    a_\eps(\theta):=\mathfrak A(b_\eps(\theta)),
    \quad \theta\geq0,
\end{align}
Extend $a_\eps$ to $(-\infty,0)$ by constancy. By construction, it is easy to see that $a_\eps\in C_b^1(\mR)$,
and
\begin{align*}
    \eps\leq b_\eps(\theta)
    \leq\eps+2\eps^{-1}\leq3\eps^{-1}.
\end{align*}
Since $\mathfrak A$ is decreasing,
\begin{align*}
    0<\mathfrak A(3\eps^{-1})
    \leq a_\eps(\theta)
    \leq\mathfrak A(\eps)<\infty,
\end{align*}
which proves \eqref{approx-uniform-ellipticity}.

By the chain rule, \eqref{b-derivative-bound},
\eqref{h-eps-properties} and \eqref{Ab8}, we also have
\begin{align}\label{b-eps-derivative-bound}
    |b_\eps'(\theta)|
    \leq C\kappa_0
    (\theta^{q_0-2}+\theta^{q_1-2}),
    \qquad \theta>0,
\end{align}
and
\begin{align*}
    -s\mathfrak B'(s)
    =\ell s^{-\ell}+\tfrac\ell3s^{-\ell/3}
    \geq\tfrac\ell3\mathfrak B(s).
\end{align*}
Since $b_\eps=\mathfrak B(a_\eps)$,
\begin{align*}
    |(\ln a_\eps)'|
    \leq\frac{3}{\ell}\frac{|b_\eps'|}{b_\eps},
    \qquad
    \frac1{b_\eps}
    \leq a_\eps^\ell\wedge a_\eps^{\ell/3}.
\end{align*}
Together with \eqref{b-eps-derivative-bound}, this proves
\eqref{approx-log-estimate}.

By \eqref{b-linear-growth} and \eqref{h-eps-properties}, we have
\begin{align*}
    a_\eps(\theta)^{-\ell}
    \vee a_\eps(\theta)^{-\ell/3}
    \leq b_\eps(\theta)
    \leq C(1+\theta),
\end{align*}
which proves \eqref{approx-integral-estimate}.

Let $K\Subset(0,\infty)$. For all sufficiently small $\eps$, we have
\begin{align*}
    b_\eps(\theta)=\eps+b(\theta)
    \quad\text{on }K,
\end{align*}
which proves \eqref{approx-local-uniform}.

If $\theta\geq2\delta_\eps$, then either
$b(\theta)\leq 1/\eps$, in which case $a_\eps(\theta)\leq a(\theta)$,
or $b(\theta)>1/\eps$, in which case
$\varpi(\eps(b(\theta)))\geq 1$ by monotonicity,
and hence $a_\eps(\theta)\leq\lambda_\eps$. Thus
\begin{align}\label{a-eps-comparison}
    a_\eps(\theta)\leq a(\theta)+\lambda_\eps,
    \qquad \theta\geq2\delta_\eps.
\end{align}
On $[0,2\delta_\eps]$, we have
$a_\eps\leq\Lambda_\eps$. Now fix $R>0$. Since $b$ is bounded on
$[0,R]$, the upper truncation is inactive there for all sufficiently small
$\eps$; in particular, $a_\eps\leq a$ on
$[2\delta_\eps,R]$. Therefore
\begin{align*}
    \int_0^R|a_\eps(\theta)-a(\theta)|\dif\theta
    &\leq2\delta_\eps\Lambda_\eps
      +\int_0^{2\delta_\eps}a(\theta)\dif\theta
      +\int_{2\delta_\eps}^R(a(\theta)-a_\eps(\theta))\dif\theta\\
    &\leq2\eps+\int_0^{2\delta_\eps}a(\theta)\dif\theta
      +\int_{2\delta_\eps}^R(a(\theta)-a_\eps(\theta))\dif\theta
    \longrightarrow0,
\end{align*}
where the last integral converges to zero by the dominated convergence theorem.
This proves \eqref{approx-local-L1}.

Splitting the integrals at $2\delta_\eps$ and using
\eqref{a-eps-comparison} gives the primitive comparisons. On
$[0,2\delta_\eps]$,
$a_\eps\leq\Lambda_\eps$, whereas on
$[2\delta_\eps,\infty)$, \eqref{a-eps-comparison} holds. Therefore
\begin{align*}
    \Phi_{a_\eps}(s)
    &\leq \Phi_a(s)+\lambda_\eps s
      +2\delta_\eps\Lambda_\eps
    \leq \Phi_a(s)+\lambda_\eps s+2\eps,
\end{align*}
which is \eqref{primitive-comparison}. 
\end{proof}
\br
In particular, for every $R>0$,
\begin{align}\label{primitive-local-convergence}
\begin{split}
    \sup_{s\in[0,R]}
    |\Phi_{a_\eps}(s)-\Phi_a(s)|&\longrightarrow0,\\
    \sup_{s\in[0,R]}
    |H_{a_\eps}(s)-H_a(s)|&\longrightarrow0.
\end{split}
\end{align}
where the second limit follows from the first one and by the Cauchy--Schwarz inequality and
$$
|\sqrt{x}-\sqrt y|^2\leq|x-y|
$$
\er

\section{Proof of the Duhamel formula}\label{appendix:Duhamel}

\begin{proof}[Proof of Lemma~\ref{Le34}]
Let
\begin{align}\label{duhamel-time-disintegration}
    \pi(\dif s,\dif z,\dif\theta)
    =\pi(s,\dif z,\dif\theta)\dif s,
\end{align}
where
\begin{align*}
    \pi(s,\dif z,\dif\theta)
    :=|\nabla_vH_a(f(s,z))|^2
      \delta_{f(s,z)}(\dif\theta)\dif z.
\end{align*}
In particular, by the energy estimate, for every $t>0$,
\begin{align*}
    \pi([0,t]\times\mR^{2d}\times\mR_+)
    =\int_0^t\!\!\int_{\mR^{2d}}
      |\nabla_vH_a(f)|^2\dif z\dif s
    \leq\frac12\|f_0\|_{L^2}^2.
\end{align*}
Since $f$ preserves mass, by Proposition~\ref{prop:kinetic-to-weak}, we also have
\begin{align}\label{duhamel-cutoff-measure}
    \lim_{\delta\downarrow0}
    \int_{[0,t]\times\mR^{2d}\times\mR_+}
    |\zeta_\delta'(\theta)|\dif\pi=0.
\end{align}
Using the fixed-time identity in
Proposition~\ref{prop:regular-renormalization}(ii), by a
standard approximation for time product rule, we have for any $\varphi\in C_c^1(\mR_+\times \mR^{2d}\times(0,\infty))$
that is $C^2$ in $z$, and $t>0$,
\begin{align}\label{duhamel-fixed-time-kinetic}
\begin{split}
\int_{\mR^{2d}}\!\int_0^\infty
    \chi(t,z,\theta)\varphi(t,z,\theta)\dif\theta\dif z
&=
\int_{\mR^{2d}}\!\int_0^\infty
    \chi_0(z,\theta)\varphi(0,z,\theta)\dif\theta\dif z\\
&+
\int_0^t\!\!\int_{\mR^{2d}}\!\int_0^\infty
    \chi\bigl(\partial_s\varphi
    +a(\theta)\Delta_v\varphi
    +v\cdot\nabla_x\varphi\bigr)
    \dif\theta\dif z\dif s\\
&-
\int_0^t\!\!\int_{\mR^{2d}\times\mR_+}
    \partial_\theta\varphi(s,z,\theta)
    \pi(s,\dif z,\dif\theta)\dif s.
\end{split}
\end{align}
Fix $t>0$ and
$g\in\cS(\mR^{2d})$. 
For $0\leq s\leq t$, set
$
    \Psi_{t,s}^\theta g:=(P_{t-s}^\theta)^*g,
$
where $(P_t^\theta)^*$ is the adjoint semigroup of $P_t$ defined by \eqref{Semi1}.
Note that
\begin{align}\label{duhamel-backward-equation}
    \partial_s\Psi_{t,s}^\theta g
    +v\cdot\nabla_x\Psi_{t,s}^\theta g
    +a(\theta)\Delta_v\Psi_{t,s}^\theta g=0,
    \qquad
    \Psi_{t,t}^\theta g=g.
\end{align}
By \eqref{Fs1} and \eqref{uniform-ellipticity}, we have for $r>0$,
\begin{align}\label{duhamel-adjoint-derivative}
    \|\partial_\theta(P_r^\theta)^*g\|_{L^\infty}
\leq
    \|\partial_\theta(\Gamma_rp_r^\theta)\|_{L^1}
    \|g\|_{L^\infty}\leq
    C_d\|g\|_{L^\infty}|(\ln a(\theta))'|
\leq C_d\|g\|_{L^\infty}L_a/\overline{a}.
\end{align}
Fix $\delta\in(0,1/4)$ and let $\phi_R(x,v):=\gamma_{R^2}(x)\gamma_R(v).$
If we take $\varphi(s,z,\theta)
    :=\phi_R(z)\zeta_\delta(\theta)
      \Psi_{t,s}^\theta g(z)$ in \eqref{duhamel-fixed-time-kinetic}, 
then using \eqref{duhamel-backward-equation}, we have
\begin{align*}
&\int_{\mR^{2d}}\!\int_0^\infty
    \chi(t,z,\theta)\phi_R(z)\zeta_\delta(\theta)g(z)\dif\theta\dif z
=
\int_{\mR^{2d}}\!\int_0^\infty
    \chi_0(z,\theta)\phi_R(z)\zeta_\delta(\theta)
      (P_t^\theta)^* g(z)\dif\theta\dif z\\
&\qquad+
\int_0^t\!\!\int_{\mR^{2d}}\!\int_0^\infty
\chi\zeta_\delta\bigl[
\Psi_{t,s}^\theta g
  (v\cdot\nabla_x\phi_R+a(\theta)\Delta_v\phi_R)
+2a(\theta)\nabla_v\phi_R\cdot
  \nabla_v\Psi_{t,s}^\theta g
\bigr]\dif\theta\dif z\dif s\\
&\qquad-
\int_0^t\!\!\int_{\mR^{2d}\times\mR_+}
\phi_R(z)    \partial_\theta[\zeta_\delta(\theta)
      \Psi_{t,s}^\theta g(z)]
    \pi(s,\dif z,\dif\theta)\dif s.
\end{align*}
Letting $R\to\infty$ and by the dominated convergence theorem, we obtain
\begin{align}\label{truncated-duhamel-duality}
\begin{split}
    \int_0^\infty\zeta_\delta(\theta)
      \langle\chi(t,\theta),g\rangle\dif\theta
    &=\int_0^\infty\zeta_\delta(\theta)
      \langle P_t^\theta\chi_0(\theta),g\rangle\dif\theta-
      \int_0^t\!\int_{\mR^{2d}\times\mR_+}
      \zeta_\delta(\theta)
      \partial_\theta\Psi_{t,s}^\theta g(z)
      \pi(s,\dif z,\dif\theta)\dif s\\
    &\quad-
      \int_0^t\!\int_{\mR^{2d}\times\mR_+}
      \zeta_\delta'(\theta)\Psi_{t,s}^\theta g(z)
      \pi(s,\dif z,\dif\theta)\dif s.
\end{split}
\end{align}
Since $\|\Psi_{t,s}^\theta g\|_{L^\infty}\leq\|g\|_{L^\infty}$ and
\begin{align*}
    \int_0^\infty|\langle\chi(t,\theta),g\rangle|\dif\theta
    \leq\|g\|_{L^\infty}\|f(t)\|_{L^1},\quad\
    \int_0^\infty
      |\langle P_t^\theta\chi_0(\theta),g\rangle|\dif\theta
    \leq\|g\|_{L^\infty}\|f_0\|_{L^1},
\end{align*} 
letting $\delta\downarrow0$ in \eqref{truncated-duhamel-duality}, by \eqref{duhamel-cutoff-measure},
\eqref{duhamel-adjoint-derivative} and 
the dominated convergence theorem, we obtain
\begin{align*}
    \langle f(t),g\rangle
    &=\int_0^\infty
    \langle P_t^\theta\chi_0(\theta),g\rangle\dif\theta-
    \int_0^t\!\!\int_{\mR^{2d}\times\mR_+}
    \partial_\theta(P_{t-s}^\theta)^*g(z)
    \pi(s,\dif z,\dif\theta)\dif s.
\end{align*}
Since 
$g\in\cS(\mR^{2d})$ are arbitrary, \eqref{Duhamel-formula} holds for
every $t\geq0$.
\end{proof}

{\bf Acknowledgements:}
This work was supported by the National Natural Science Foundation of
China, Grant No.~12595282, and the National Key R\&D Program of China,
Grant No.~2023YFA1010103.

\vspace{5mm}

{\bf Disclosure of AI:}
We used ChatGPT 5.6 Pro to produce the proof of Lemma A.1 and to create the picture in Figure 1. We also used AI tools (such as Kimi K3 and ChatGPT 5.6 Pro) to collect references, polish the language, and rephrase the statements. All AI-assisted edits were reviewed and approved by the authors, who take full responsibility for the final version.

% Local alpha style: uniform author initials and DOI/arXiv links.
%\bibliographystyle{kpm-alpha}
%\bibliography{reference-kpm-revised}
\bibliographystyle{plain}
\bibliography{reference-kpm-revised}

@misc{BCD26,
      title={The fundamental solution of a nonlinear kinetic {F}okker--{P}lanck equation},
      author={Giovanni Brigati and Guillaume Carlier and Jean Dolbeault},
      year={2026},
      eprint={2603.26650},
      archivePrefix={arXiv},
      primaryClass={math.AP},
      url={https://arxiv.org/abs/2603.26650},

}

@article{K72,
      title={{Schr\"odinger} operators with singular potentials},
      author={Tosio Kato},
      year={1972},
      volume = {13},
      journal = {Israel J. Math.},
      number = {1--2},
      DOI = {10.1007/BF02760233},
      pages = {135--148},
}

@misc{BCDQ26,
      title={Nonlinear kinetic {F}okker--{P}lanck equations as gradient flows of the free energy},
      author={Giovanni Brigati and Guillaume Carlier and Jean Dolbeault and Filippo Quattrocchi},
      year={2026},
      eprint={2606.16315},
      archivePrefix={arXiv},
      primaryClass={math.AP},
      url={https://arxiv.org/abs/2606.16315},

}

@article {BRS19,
    AUTHOR = {Bogachev, Vladimir I. and R\"ockner, Michael and Shaposhnikov,
              Stanislav V.},
     TITLE = {On the {A}mbrosio--{F}igalli--{T}revisan superposition principle
              for probability solutions to {F}okker--{P}lanck--{K}olmogorov
              equations},
   JOURNAL = {J. Dynam. Differential Equations},
  FJOURNAL = {Journal of Dynamics and Differential Equations},
    VOLUME = {33},
      YEAR = {2021},
    NUMBER = {2},
     PAGES = {715--739},
      ISSN = {1040-7294,1572-9222},
   MRCLASS = {60J60 (35Q84)},
  MRNUMBER = {4248630},
MRREVIEWER = {Giacomo\ Ascione},
       DOI = {10.1007/s10884-020-09828-5},
       URL = {https://doi.org/10.1007/s10884-020-09828-5},
}

@article{BR20,
author = {Viorel Barbu and Michael R\"ockner},
title = {From nonlinear {F}okker--{P}lanck equations to solutions of distribution dependent {SDE}},
volume = {48},
journal = {Ann. Probab.},
number = {4},
publisher = {Institute of Mathematical Statistics},
pages = {1902--1920},
year = {2020},
doi = {10.1214/19-AOP1410},
URL = {https://doi.org/10.1214/19-AOP1410}
}

@book {V07,
    AUTHOR = {V{\'a}zquez, Juan Luis},
     TITLE = {The porous medium equation},
    SERIES = {Oxford Mathematical Monographs},
      NOTE = {Mathematical theory},
 PUBLISHER = {The Clarendon Press, Oxford University Press, Oxford},
      YEAR = {2007},
     PAGES = {xxii+624},
      ISBN = {978-0-19-856903-9; 0-19-856903-3},
   MRCLASS = {35-02 (35B05 35B40 35K57 76S05)},
  MRNUMBER = {2286292},
MRREVIEWER = {Vicen\c tiu\ D.\ R\u adulescu},
}

@book {DK07,
    AUTHOR = {Daskalopoulos, Panagiota and Kenig, Carlos E.},
     TITLE = {Degenerate diffusions},
    SERIES = {EMS Tracts in Mathematics},
    VOLUME = {1},
      NOTE = {Initial value problems and local regularity theory},
 PUBLISHER = {European Mathematical Society (EMS), Z\"urich},
      YEAR = {2007},
     PAGES = {x+198},
      ISBN = {978-3-03719-033-3},
   MRCLASS = {35K57 (35-02 35B40 35B41 35B45)},
  MRNUMBER = {2338118},
MRREVIEWER = {Roberto\ Natalini},
       DOI = {10.4171/033},
       URL = {https://doi.org/10.4171/033},
}

@article {HP85,
    AUTHOR = {Herrero, Miguel A. and Pierre, Michel},
     TITLE = {The {C}auchy problem for {$u_t=\Delta u^m$} when {$0<m<1$}},
   JOURNAL = {Trans. Amer. Math. Soc.},
  FJOURNAL = {Transactions of the American Mathematical Society},
    VOLUME = {291},
      YEAR = {1985},
    NUMBER = {1},
     PAGES = {145--158},
      ISSN = {0002-9947,1088-6850},
   MRCLASS = {35K55 (76X05)},
  MRNUMBER = {797051},
MRREVIEWER = {L.\ Hsiao},
       DOI = {10.2307/1999900},
       URL = {https://doi.org/10.2307/1999900},
}

@article {LPT94,
    AUTHOR = {Lions, P.-L. and Perthame, B. and Tadmor, E.},
     TITLE = {A kinetic formulation of multidimensional scalar conservation
              laws and related equations},
   JOURNAL = {J. Amer. Math. Soc.},
  FJOURNAL = {Journal of the American Mathematical Society},
    VOLUME = {7},
      YEAR = {1994},
    NUMBER = {1},
     PAGES = {169--191},
      ISSN = {0894-0347,1088-6834},
   MRCLASS = {35L65 (35K65)},
  MRNUMBER = {1201239},
MRREVIEWER = {Long\ Wei\ Lin},
       DOI = {10.2307/2152725},
       URL = {https://doi.org/10.2307/2152725},
}

@article {Ko34,
    AUTHOR = {Kolmogoroff, A.},
     TITLE = {Zuf\"allige {B}ewegungen (zur {T}heorie der {B}rownschen
              {B}ewegung)},
   JOURNAL = {Ann. of Math. (2)},
  FJOURNAL = {Annals of Mathematics. Second Series},
    VOLUME = {35},
      YEAR = {1934},
    NUMBER = {1},
     PAGES = {116--117},
      ISSN = {0003-486X,1939-8980},
   MRCLASS = {99-04},
  MRNUMBER = {1503147},
       DOI = {10.2307/1968123},
       URL = {https://doi.org/10.2307/1968123},
}

@article {HZZZ24,
    AUTHOR = {Hao, Zimo and Zhang, Xicheng and Zhu, Rongchan and Zhu,
              Xiangchan},
     TITLE = {Singular kinetic equations and applications},
   JOURNAL = {Ann. Probab.},
  FJOURNAL = {The Annals of Probability},
    VOLUME = {52},
      YEAR = {2024},
    NUMBER = {2},
     PAGES = {576--657},
      ISSN = {0091-1798,2168-894X},
   MRCLASS = {60H15 (35K15 35R09 35R60)},
  MRNUMBER = {4718402},
MRREVIEWER = {Pedro\ Mar\'in Rubio},
       DOI = {10.1214/23-aop1666},
       URL = {https://doi.org/10.1214/23-aop1666},
}

@article {HRZ25,
    AUTHOR = {Hao, Z. and R\"ockner, M. and Zhang, X.},
     TITLE = {Second-order fractional mean-field {SDE}s with singular
              kernels and measure initial data},
   JOURNAL = {Ann. Probab.},
  FJOURNAL = {The Annals of Probability},
    VOLUME = {54},
      YEAR = {2026},
    NUMBER = {1},
     PAGES = {1--62},
      ISSN = {0091-1798,2168-894X},
   MRCLASS = {60H10 (35R11 35R60)},
  MRNUMBER = {5019007},
       DOI = {10.1214/24-AOP1709},
       URL = {https://doi.org/10.1214/24-AOP1709},
}

@article {Bo02,
    AUTHOR = {Bouchut, F.},
     TITLE = {Hypoelliptic regularity in kinetic equations},
   JOURNAL = {J. Math. Pures Appl. (9)},
  FJOURNAL = {Journal de Math\'ematiques Pures et Appliqu\'ees. Neuvi\`eme
              S\'erie},
    VOLUME = {81},
      YEAR = {2002},
    NUMBER = {11},
     PAGES = {1135--1159},
      ISSN = {0021-7824},
   MRCLASS = {82C40 (35B65 35H10)},
  MRNUMBER = {1949176},
MRREVIEWER = {Florent\ Berthelin},
       DOI = {10.1016/S0021-7824(02)01264-3},
       URL = {https://doi.org/10.1016/S0021-7824(02)01264-3},
}

@article {HWZ20,
    AUTHOR = {Hao, Zimo and Wu, Mingyan and Zhang, Xicheng},
     TITLE = {{S}chauder estimates for nonlocal kinetic equations and
              applications},
   JOURNAL = {J. Math. Pures Appl. (9)},
  FJOURNAL = {Journal de Math\'ematiques Pures et Appliqu\'ees. Neuvi\`eme
              S\'erie},
    VOLUME = {140},
      YEAR = {2020},
     PAGES = {139--184},
      ISSN = {0021-7824,1776-3371},
   MRCLASS = {35K65 (35K08 35R09 35R60 60H10)},
  MRNUMBER = {4124429},
MRREVIEWER = {J\o rgen\ Endal},
       DOI = {10.1016/j.matpur.2020.06.003},
       URL = {https://doi.org/10.1016/j.matpur.2020.06.003},
}

@article {IS21,
    AUTHOR = {Imbert, Cyril and Silvestre, Luis},
     TITLE = {The {S}chauder estimate for kinetic integral equations},
   JOURNAL = {Anal. PDE},
  FJOURNAL = {Analysis \& PDE},
    VOLUME = {14},
      YEAR = {2021},
    NUMBER = {1},
     PAGES = {171--204},
      ISSN = {2157-5045,1948-206X},
   MRCLASS = {35K70 (35R09)},
  MRNUMBER = {4229202},
       DOI = {10.2140/apde.2021.14.171},
       URL = {https://doi.org/10.2140/apde.2021.14.171},
}

@article {CHM21,
    AUTHOR = {Chaudru de Raynal, Paul-\'Eric and Honor\'e, Igor and Menozzi,
              St\'ephane},
     TITLE = {Sharp {S}chauder estimates for some degenerate {K}olmogorov
              equations},
   JOURNAL = {Ann. Sc. Norm. Super. Pisa Cl. Sci. (5)},
  FJOURNAL = {Annali della Scuola Normale Superiore di Pisa. Classe di
              Scienze. Serie V},
    VOLUME = {22},
      YEAR = {2021},
    NUMBER = {3},
     PAGES = {989--1089},
      ISSN = {0391-173X,2036-2145},
   MRCLASS = {35B65 (35R60 60H10)},
  MRNUMBER = {4334312},
       DOI = {10.2422/2036-2145.201811_011},
       URL = {https://doi.org/10.2422/2036-2145.201811_011},
}

@misc{BDM26,
      title={Nonlinear kinetic {F}okker--{P}lanck equations: existence and diffusion limits},
      author={{\'E}meric Bouin and Jean Dolbeault and Antoine Mellet},
      year={2026},
      eprint={2606.31899},
      archivePrefix={arXiv},
      primaryClass={math.AP},
      url={https://arxiv.org/abs/2606.31899},

}

@book {P02,
    AUTHOR = {Perthame, Beno\^{i}t},
     TITLE = {Kinetic formulation of conservation laws},
    SERIES = {Oxford Lecture Series in Mathematics and its Applications},
    VOLUME = {21},
 PUBLISHER = {Oxford University Press, Oxford},
      YEAR = {2002},
     PAGES = {xii+198},
      ISBN = {0-19-850913-8},
   MRCLASS = {35-02 (35B25 35F20 35L65 76N10 76N15 76P05 82C40)},
  MRNUMBER = {2064166},
MRREVIEWER = {Manuel Portilheiro},
}

\end{document}